\documentclass[12pt,reqno]{amsart}
\usepackage{amsmath, amsthm, amsopn, amssymb, microtype}
\usepackage{mathrsfs} 
\usepackage{mathscinet}
\usepackage{bbm}
\usepackage{enumerate, tikz, etoolbox, intcalc, geometry, caption, subcaption, afterpage}
\usepackage[english]{babel}
\usepackage{enumitem}

\usepackage[colorlinks=true,urlcolor=blue,pdfborder={0 0 0}]{hyperref}
\hypersetup{linkcolor=[rgb]{0,0,0.6}}
\hypersetup{citecolor=[rgb]{0,0.6,0}}
\usepackage[nameinlink]{cleveref}
\crefname{theorem}{Theorem}{Theorems}
\crefname{thm}{Theorem}{Theorems}
\crefname{conj}{Conjecture}{Theorems}
\crefname{lemma}{Lemma}{Lemmas}
\crefname{lem}{Lemma}{Lemmas}
\crefname{remark}{Remark}{Remarks}
\crefname{prop}{Proposition}{Propositions}
\crefname{defn}{Definition}{Definitions}
\crefname{corollary}{Corollary}{Corollaries}
\crefname{cor}{Corollary}{Corollaries}
\crefname{section}{Section}{Sections}
\crefname{figure}{Figure}{Figures}
\crefname{quest}{Question}{Questions}

\newtheorem{theorem}{Theorem}[section]
\newtheorem{lemma}{Lemma}[section]
\newtheorem{corollary}{Corollary}[section]
\newtheorem{proposition}{Proposition}[section]

\theoremstyle{definition}

\DeclareMathOperator{\homeo}{Homeo}
\DeclareMathOperator{\Prob}{Prob}

\DeclareMathOperator{\mesh}{mesh}
\DeclareMathOperator{\shift}{shift}

\DeclareMathOperator{\Fix}{Fix}
\DeclareMathOperator{\ord}{ord}

\DeclareMathOperator{\mdim}{mdim}
\DeclareMathOperator{\Cov}{Cov}
\DeclareMathOperator{\capacity}{cap}
\newcommand{\Z}{\mathbb{Z}}

\newcommand{\N}{\mathbb{N}}
\newcommand{\R}{\mathbb{R}}

\title[Dimension for dynamical systems]{A new notion of dimension for dynamical systems and shift embeddability}
\author{Tom Meyerovitch }
\address{Ben Gurion University of the Negev.
	Departement of Mathematics.
	Be'er Sheva, 8410501, Israel
}
\email{mtom@bgu.ac.il}

\begin{document}

\begin{abstract}
    A dynamical system $(X,T)$ is \emph{shift embeddable} if $(X,T)$ embeds continuously and equivariantly in the shift over $[0,1]^d$ for some finite $d$. Refuting a major conjecture in the field, in a recent result of Dranishnikov and Levin it was shown that Gromov's mean dimension and Lebesgue covering dimension of finite orbits are not the only obstructions for shift embeddability.
    We present a new notion of dimension for dynamical systems over any countable group. We show that this new notion of dimension accounts for all known obstructions for shift embeddability.
\end{abstract}

\maketitle

\section{Introduction}\label{sec:intro}


Let $\Phi:X\to X$ be a homeomorphism of a compact metric space. 
The \emph{shift embeddability} or \emph{sampling-rate problem} is the following:
Under what conditions do there exists a finite number of  continuous real-valued functions $f_1,\ldots,f_d:X \to \mathbb{R}$ so that a point  $x \in X$ can be uniquely recovered by sampling the values of the $f_1,\ldots,f_d$ along the orbit of $x$?
Formally, the shift embeddability problem for $\Phi:X \to X$ asks if there exists $d \in \mathbb{N}$ and a continuous function $f:X \to \mathbb{R}^d$ such that the function $f^\mathbb{Z}:X \to (\mathbb{R}^d)^\Z$ is injective, where $f^\mathbb{Z}$ is given by $f^\mathbb{Z}(x)_n: =f(\Phi^n(x))$.
If there exists $d \in \N$ and $f:X \to \mathbb{R}^d$ as above, we say that $\Phi$ is \emph{shift embeddable}, and in this case refer to the smallest possible $d \in \N$  as the \emph{minimal sampling rate} for $\Phi$.

In the presence of finite orbits, classical dimension theory immediately confronts us with obstructions: Shift embeddablity of $\Phi$ implies that  for  every $k \in \mathbb{N}$ the  set of $k$-periodic points must be finite dimensional. In fact, embeddablity of $\Phi$ implies  
\begin{equation}\label{eq:fix_finite}
\sup_{k \in \N}\frac{1}{k}\dim(\Fix(\Phi^k))< +\infty
\end{equation}

Under the additional assumption that the dimension of the space $X$ is finite, 
it was shown by Gutman, extending  Jawarski's previous work, that there are no additional obstructions to shift embeddablity beyond \eqref{eq:fix_finite}.

For many years it was not known if there are additional obstructions for shift embeddablity of a dynamical system.
Near the end of the 20th century Lindenstrauss and Weiss \cite{MR1749670}  discovered that   \emph{mean dimension} poses such an obstruction. The mean dimension of a homeomorphism is a numerical invariant introduced by Gromov that can take arbitrary values in $[0,+\infty]$. 
We will recall a definition of mean dimension later. For now let us say that it can be thought of as a crossbreeding of Lebesgue covering dimension of a topological space and the and topological entropy of a continuous self-map.
 Lindenstrauss and Weiss showed that if $\Phi$ has minimal sampling rate $d$, then the mean dimension of $\Phi$ cannot exceed $d$, and that for any $ d\in [0,+\infty]$ there exist a minimal homeomorphism $\Phi:X \to X$ whose mean dimension is equal to $d$. In a subsequent paper \cite{MR1793417}  Lindenstrauss proved that for a class of homeomorphisms that includes all minimal homeomorphism and is closed under extensions, finite mean dimension is a necessarily and sufficient condition for finiteness of the minimal sampling rate.

The discovery of mean dimension and the embedding theorem for minimal systems due to  Lindenstrauss \cite{MR1793417}
recast the shift embeddability problem into the following form: Beyond mean dimension and dimension of  periodic points, are there any additional obstructions to shift embeddability? 

Up until recently, it was  conjectured that no additional obstructions exist, see  \cite[Conjecture 1.2]{MR3219549}.  
In a recent paper \cite{dranishnikov2025freezactionisometriescompact} Dranishnikov and Levin shattered the above conjecture by proving the existence of an \emph{isometry} on a compact metric space $X$ that 
cannot be embedded into the shift over $[0,1]^d$ for any finite $d$ .

Could there be a different ``dimension-like'' invariant that completely captures all obstructions for shift embeddability?

With the above question in mind, in the current paper we propose a new notion of dimension for dynamical systems.

To avoid further suspense, we will now present a concise definition of this new invariant.
Let $(X,T)$ be a dynamical system, namely $X$ is a compact metric space and $T$ is a homomorphism from a (discrete countable or locally compact metrizable) group $\Gamma$ into the group of homeomorphism  of $X$.
Let  $\Cov(X)$ be the space of open covers of $X$, let $\Prob(X,T)$ be the space of $T$-invariant Borel probability measures on $X$. Write $\mathcal{V} \succeq \mathcal{U}$ to indicate  that the cover $\mathcal{V}$ refines $\mathcal{U}$ in the sense that for any $V \in \mathcal{V}$ there exists $U \in \mathcal{U}$ such that $V \subseteq U$.
We define the \emph{dimension} of $(X,T)$ as follows:

\begin{equation}\label{dim_basic_def}
\dim(X,T) = \sup_{\mathcal{U} \in \Cov(X)}\inf_{\mathcal{V} \succeq \mathcal{U}}\sup_{\mu \in \Prob(X,T)} \int_{X} \left(-1+\sum_{V \in \mathcal{V}}\mathbf{1}_V(x) \right)d\mu(x).    
\end{equation}

To have a direct comparison in mind, we remind the reader that Gromov's mean dimension for a dynamical system $(X,T)$, where the acting group $\Gamma$ is  amenable can be written as follows:

\begin{equation}\label{eq:mean_dim_def}
    \mdim(X,T) = \sup_{\mathcal{U} \in \Cov(X)}\lim_{n \to \infty}\frac{1}{|F_n|}\inf_{\mathcal{V} \succeq \bigvee_{\gamma \in F_n}T_\gamma(\mathcal{U})}\sup_{x \in X} \left(-1+\sum_{V \in \mathcal{V}}\mathbf{1}_V(x) \right),
\end{equation}
where $(F_n)_{n=1}^\infty$ is a F\o lner sequence in $\Gamma$ and $\vee_{\gamma \in F_n}T_\gamma(\mathcal{U})$ is the joint refinement of the iterates of $\mathcal{U}$ under elements of $F_n$, namely the cover consisting of elements of the form $\bigcap_{\gamma \in F_n}T_\gamma(U_\gamma)$, where for every $\gamma \in F_n$ $U_\gamma$ is some element of the cover $\mathcal{U}$.

The reader can readily check that in the case where the action $T$ is the trivial action we have $\dim(X,T)$ coincides with Lebesgue covering dimension of $X$, whereas the $\mdim(X,T)=0$ for the trivial action of any infinite amenable group $\Gamma$ on a compact metric space $X$. So at least in trivial cases the new notion of dimension does not coincide with Gromov's mean dimension.

We list some  features of this new invariant:

\begin{enumerate}
    \item The quantity $\dim(X,T)$ is monotonic with respect to embeddings of dynamical system. By this we mean that  if a $(X,T)$  and $(Y,S)$ are dynamical systems with the same acting group $\Gamma$, then $\dim(X,T) \le \dim(Y,S)$. A short proof of this fact appears as \Cref{prop:dim_emb_monotone}.
    \item In every context where the mean dimension of $(X,T)$ has been defined we have $\mdim(X,T) \le \dim(X,T)$. Mean dimension of a dynamical system has been defined  when the acting group is amenable, or more generally using Li's sofic mean dimension when the acting group is sofic \cite{MR3077882}. We prove that $\dim(X,T)$ is an upper bound for Gromov's mean dimension in \Cref{thm:mdim_le_dim}. The corresponding proof for the more general sofic case follows similar ideas, but requires introducing  additional terminology  in order to define  sofic groups and sofic mean dimension, so we do not present it here.
    \item For any $d \in \mathbb{N}$ and \emph{any} countable group $\Gamma$ we have $\dim(([0,1]^d)^\Gamma,\shift)=d$. We prove this fact as \Cref{thm:dim_cubical_shift}.
    \item In all known cases where finiteness of mean dimension has been confirmed to be a sufficient condition for shift embeddability, we can show that $\dim(X,T) = \mdim(X,T)$. In particular, $\dim(X,T)=\mdim(X,T)$ for any   $\Z^k$-dynamical system with the marker property. In particular, the new invariant coincides with Gromov's mean dimension for any $\Z$-system with  free minimal factor or with a free finite dimensional factor. A proof of this result appears as \Cref{thm:mdim_equal_dim_marker_property}.
    \item If $(X,T)$ is a dynamical system with acting group $\Gamma$ and $\Gamma_0$ is a finite-index subgroup of $\Gamma$, then $\dim(X,T) \ge \frac{1}{[\Gamma:\Gamma_0]}\dim(\Fix_H(X))$, so the invariant $\dim(X,T)$ already accounts for the periodic point constraints that arise from classical dimension theory.
    \item For some classes of dynamical systems, including all distal  dynamical systems, we can prove that finiteness of  $\dim(X,T) $ is a necessary and sufficient condition for shift embeddability. We state and prove this as  \Cref{cor:distal_emb}. Since an isometry is distal, we conclude that the recent Dranishnikov-Levin counterexamples must have infinite dimension in the sense of our new definition, so even for free $\Z$-dynamical systems it can happen that $\mdim(X,T)=0$ and $\dim(X,T)=\infty$. This is stated as  \Cref{cor:dim_ne_mdim}. We emphasize that our conclusion depends upon and does not replace the intricate arguments of \cite{dranishnikov2025freezactionisometriescompact}, which make use of the proof  of the equivariant Sullivan conjecture.   
    \item For any dynamical system $\dim(X,T)=0$ if and only if $(X,T)$ has the small boundary property. This is stated as \Cref{thm:SBP_iff_dim_0}. 
    \item The dimension of a dynamical system is an invariant of topological orbit equivalence. In fact it is an invariant of a weaker notion of equivalence that we call topological-measure pairs isomorphism. As a consequence we have that at least  for free minimal dynamical systems where the acting group is  $\mathbb{Z}^k$, mean dimension is an invariant of topological orbit equivalence.
    \item For any dynamical system finiteness of $\dim(X,T)$ is a  sufficient condition for \emph{almost shift embeddability}.  By almost shift embeddability we mean that there exists a continuous equivariant map into the shift over $[0,1]^d$ for some $d \in \mathbb{N}$ which induces a measure-theoretic isomorphism with respect to any invariant measure.  See \Cref{sec:almost_emb}.
\end{enumerate}

The rest of the paper is organized as follows: In \Cref{sec:ord_dim_cover} we recall basic results from classical dimension theory. Readers who are acquainted with classical dimension theory should find the content of this section quite familiar, although the treatment of the order of a cover as a function rather than a number is perhaps not completely standard.  In \Cref{sec:TM_pairs} we discuss the category of topological-measure pairs following Elliott and Niu \cite{elliott2025smallboundarypropertymathcal}  and define a notion of dimension in this category. In \Cref{sec:dynamical_systems} we briefly observe that any dynamical system gives a topological-measure pair.
This extends the notion of dimension to the category of dynamical systems.
In \Cref{sec:SBP} we recall the small boundary property and prove that a dynamical system has the small boundary property if and only if it has zero dimension in the sense of our new definition. In \Cref{sec:Baire_Category} we formulate and prove a technical Baire-category result that  is applied in later sections. In \Cref{Sec:finite_groups} we explain how the dimension of dynamical systems over finite groups can be expressed in terms of Lebesgue covering dimension. In \Cref{sec:dim_cube_shift} we prove that the dimension of the shift over $[0,1]^d$ over any countable group is equal to $d$.
In \Cref{sec:dim_amenable} we discuss the new notion of dimension in the case where the acting group is amenable. We prove that in this case our notion of dimension is always an upper bound for Gromov's mean dimension and prove that the two notion coincide in the case where the action has the uniform Rokhlin property. In \Cref{sec:almost_emb} we define a notion of almost embedding for dynamical systems and show that any dynamical system with dimension strictly less than $d/2$ almost embeds in the shift over $[0,1]^d$. As a corollary, we show deduce a sharp embedding theorem for distal dynamical systems over any countable group, and deduce that the Dranishnikov-Levin counterexample has infinite dimension in our sense.

I thank the anonymous referees for careful review and helpful comments.
I thank Misha Levine for motivating discussions and background in dimension theory.
I would like to deeply acknowledge Ilan Hirshberg and Chris Phillips who brought the work \cite{elliott2025smallboundarypropertymathcal} to my attention and asked the following insightful question that led me to consider the new notion of dimension: In the case of minimal $\Z$-systems, is mean dimension an invariant of topological measure pairs?

This research was partially
supported by the Israel Science Foundation grant No. 985/23.

\section{Order and dimension of covers}\label{sec:ord_dim_cover}

In this section, we recall classical notions involved in the definition of Lebesgue covering dimension. We also introduce notation and formulate auxiliary lemmas that will be applied in the proofs of subsequent results.

Let $X$ be a compact metric space.
Given a collection $\mathcal{V}$ of subsets of $X$ and $x \in X$, we denote 
\begin{equation}\label{eq:ord_at_x_def}
\ord(\mathcal{V},x) = -1+\sum_{V \in \mathcal{V}}\mathbf{1}_V(x).    
\end{equation}

We note that if $\mathcal{V}$ is a \emph{cover} of $X$ then $\ord(\mathcal{V},x) \ge 0$ for every $x \in X$.

For a collection $\mathcal{V}$ of pairwise disjoint pairwise disjoint subsets of $X$ we define $\mathbf{1}_{\mathcal{V}}:X \to \{0,1\}$ and $\overline{\mathbf{1}}_{\mathcal{V}}:X \to \{0,1\}$ by:
\begin{equation}
    \mathbf{1}_{\mathcal{V}}(x) := \sum_{V \in \mathcal{V}}\mathbf{1}_V(x),~x\in X
\end{equation}
and
\begin{equation}\label{eq:bar_ind_def}
    \overline{\mathbf{1}}_{\mathcal{V}}(x):= 1-\sum_{V \in \mathcal{V}}1_V(x) , ~ x\in X.
\end{equation}
Note that  if $\mathcal{V}$  is a collection   of pairwise disjoint subsets then $\ord(\mathcal{V},x) =  - \overline{\mathbf{1}}_{\mathcal{V}}(x)\in \{-1,0\}$ for every $x \in X$.

We write $\ord(\mathcal{V})= \sup_{x \in X} \ord(\mathcal{V},x)$.

Given two collections $\mathcal{V}$ and $\mathcal{U}$ of subsets of $X$, we say that  $\mathcal{V}$ \emph{refines}  $\mathcal{U}$, denoted by  $\mathcal{U} \preceq \mathcal{V}$, if every element of   $\mathcal{U}$ is contained in some element $\mathcal{V}$. 

We denote by $\Cov(X)$ the set of finite open covers of $X$ and by $\overline{\Cov}(X)$ the set of finite closed covers of $X$.
The \emph{dimension} of a cover $\mathcal{U}$ of $X$ is defined as follows:
\[
\dim(\mathcal{U}) = \inf_{\mathcal{V} \in \Cov(X)~:~ \mathcal{V} \succeq \mathcal{U}}\ord(\mathcal{V}).
\]
We recall that the Lebesgue covering dimension of $X$ is defined by \[\dim(X)=\sup_{\mathcal{U} \in \Cov(X)}\dim(\mathcal{U}).\]

Version of the  following basic lemma appear in many introductory topology textbooks:
\begin{lemma}\label{lem:smaller_cover_munkres}
Let $\{ U_1',\ldots,U_\ell'\}$ be a finite open cover of $X$. Then there exist open sets $U_1,\ldots, U_\ell$ so that $\overline{U_j} \subseteq U'_j$ and so that $\mathcal{U}=\left\{U_1,\ldots,U_\ell\right\}$ is still a cover of $X$.
\end{lemma}
For a proof see for instance \cite[Theorem 36.2]{MR464128}.

From \Cref{lem:smaller_cover_munkres} it follows directly that for any $\mathcal{U} \in \Cov(X)$ we have
\[
\dim(\mathcal{U}) = \inf_{\overline{\mathcal{V}} \in \overline{\Cov}(X)~:~ \overline{\mathcal{V}} \succeq \mathcal{U}}\ord(\mathcal{V}).
\]

Given two covers $\mathcal{U}_1,\mathcal{U}_2 \in \Cov(X)$, the joint refinement is given by
\[\mathcal{U}_1 \vee \mathcal{U}_2 := \left\{ 
U_1 \cap U_2~:~ U_1 \in \mathcal{U}_1 ,~ U_2 \in \mathcal{U}_2\right\}.\]

A set  $A \subset \mathbb{R}^d$ is said to be \emph{locally finite} if $A \cap K$ is finite for every compact set $K \subset \R^d$. 
Similarly a cover $\mathcal{C}$ of $\mathcal{R}^d$ is locally finite if for every compact $K \subset \R^d$ we have that $\mathcal{C}_K:= \{C \cap K~:~ C \in \mathcal{C}\}$ is a finite cover of $K$.

Given a family of sets $\mathcal{U}$ in a metric space $X$ we denote by $\mesh(\mathcal{U})$ the supremum of the diameter of the elements of $\mathcal{U}$.

The  following elementary lemma  expresses specific properties of certain  ``brickwall covers'' of  $\mathbb{R}^d$ that will be used in the proof of subsequent results:

\begin{lemma}[The ``Brickwall'' Cover]\label{lem:cube_cover_R_d}
For any $d \in \N$ and any $\epsilon >0$ there exists a locally finite cover $\mathcal{C}$ of $\mathbb{R}^d$ by closed sets with $\mesh(\mathcal{C}) \le \epsilon$ and pairwise disjoint locally finite sets $A_1,\ldots, A_d \subset \mathbb{R}$ such that the following holds:
    \begin{equation}\label{eq:ord_C_d_bound}
        \ord(\mathcal{C},v) \le \sum_{\ell =1}^d\mathbf{1}_{A_\ell}(v_\ell) ~\forall
        v=(v_1,\ldots,v_d) \in \mathbb{R}^d.
    \end{equation}

\end{lemma}
\begin{proof} 
    We will consider $\mathbb{R}^d$ with respect to the $\|\cdot\|_\infty$ norm  so for $v=(v_1,\ldots,v_d) \in \mathbb{R}^d$ and $w=(w_1,\ldots,w_d)\in \mathbb{R}^d$ the distance between $v$ and $w$  is equal to  $\|v-w\|_\infty= \max_{1\le \ell \le d}|v_i-w_i|$ (the statement of the lemma does not depend on the choice of norm).
    Fix $\epsilon >0$.
    We will prove the statement of the lemma  by induction on $d$.
    For $d=1$, we can choose $A= \frac{\epsilon}{2}\mathbb{Z}$ and $\mathcal{C}= \{[(\epsilon/2) n,(\epsilon/2)(n+1)]~:~ n \in \mathbb{Z}\}$. Then $A$ is clearly locally finite and $\mathcal{C}$ is a cover of $\mathbb{R}$ by closed sets (intervals). Since the intervals in $\mathcal{C}$ only intersect in their endpoints, which are exactly the set $A$, we have $\ord(\mathcal{C},v) = \mathbf{1}_A(v) $ for any $v \in \R$.

    Assume by induction that we found  pairwise disjoint locally finite sets $A_1,\ldots,A_d \subset \mathbb{R}$ and a cover $\mathcal{C}$ of $\mathbb{R}^d$ by closed sets with $\mesh(\mathcal{C}) \le \epsilon$ so that \Cref{eq:ord_C_d_bound} holds.

    For every $1 \le \ell \le d$  the set $A_\ell$ is locally finite, hence countable. So there exists $t \in \mathbb{R}$ such that  $A_i \cap (t+A_j)=\emptyset$ for every $1\le i,j \le d$.
    Choose such $t$ and define for every $1\le \ell \le d$,  $A_\ell' :=A_\ell \cup (t+A_\ell)$. Then the sets $A_1',\ldots,A'_d$ are locally finite and pairwise disjoint.
Furthermore, because $A:= \bigcup_{\ell=1}^d A_\ell'$ is a countable set, there exists $s \in \mathbb{R}$ such that $A \cap (s+\frac{\epsilon}{2}\mathbb{Z})=\emptyset$. Let $A'_{d+1} := (s+\frac{\epsilon}{2}\mathbb{Z})$.

Then $A_1',\ldots,A'_{d+1} \subset \mathbb{R}$ are pairwise disjoint locally finite sets.
Define:
\[
w = (\underbrace{t,\ldots,t}_d) \in \mathbb{R}^d,
\]
\[
\mathcal{C}' = \left\{
C \times [s+\epsilon n,s+\epsilon n+ \epsilon/2]~:~ C \in \mathcal{C}_d,~ n \in \Z
\right\}\]
\[ \mathcal{C}''= \ \left\{ 
(w+C)\times [s+\epsilon n + \frac{\epsilon}{2},s+\epsilon(n+1)]~:~ C \in \mathcal{C}_d,~ n \in \Z
\right\}.
\]
and 
\[ \mathcal{C}^* = \mathcal{C}' \cup \mathcal{C}''. \]
Then $\mathcal{C}^*$ is a cover of $\mathbb{R}^{d+1}$ by closed sets, and $\mesh(\mathcal{C}^*)\le \epsilon$.
Let us check that for any $v=(v_1,\ldots,v_{d+1}) \in \mathbb{R}^{d+1}$ we have
\[
\ord(\mathcal{C}^*,v) \le \sum_{\ell=1}^{d+1}\mathbf{1}_{A_\ell'}(v_\ell).
\]
Indeed, if there exists $n \in \mathbb{Z}$ such that $v_{d+1} \in (s+\epsilon n,s+\epsilon n+\frac{\epsilon}{2})$, then 
\[
\ord(\mathcal{C}^*,v) =\ord(\mathcal{C},(v_1,\ldots,v_d)) \le \sum_{\ell=1}^d \mathbf{1}_{A_\ell}(v_\ell) \le \sum_{\ell=1}^{d+1} \mathbf{1}_{A_\ell'}(v_\ell) .
\]
If there exists $n \in \mathbb{Z}$ such that $v_{d+1} \in (s+\epsilon n +\frac{\epsilon}{2},s+\epsilon(n+1))$, then
\[
\ord(\mathcal{C}^*,v) =\ord(\mathcal{C},(v_1-t,\ldots,v_d-t)) \le \sum_{\ell=1}^d \mathbf{1}_{t+A_\ell}(v_\ell) \le \sum_{\ell=1}^{d+1} \mathbf{1}_{A_\ell'}(v_\ell) .
\]
Otherwise, there exists $n \in \mathbb{Z}$ such that $v_{d+1}=\frac{\epsilon}{2}n$. In this case $v_{d+1} \in A_{d+1}'$.
Let $v'=(v_1,\ldots,v_d) \in \mathbb{R}^d$ . 
\[ \ord(\mathcal{C}^*,v)=1+\ord(\mathcal{C},v')+ \ord(\mathcal{C},v'-w)\le\]
\[1+ 
\sum_{\ell=1}^d\mathbf{1}_{A_\ell}(v_\ell)+\sum_{\ell=1}^d\mathbf{1}_{A_\ell}(v_\ell-t)=
\]
\[
= 1+ \sum_{\ell=1}^d\left( \mathbf{1}_{A_\ell}(v_\ell)+\mathbf{1}_{t+A_\ell}(v_\ell)\right)=
\]
\[
=\mathbf{1}_{A_{d+1}'}(v_{d+1})+ \sum_{\ell=1}^d\mathbf{1}_{A'_\ell}(v_\ell)=\sum_{\ell=1}^{d+1}\mathbf{1}_{A'_\ell}(v_\ell).\]
In the last equality we used that $\mathbf{1}_{A'_\ell}(v_\ell)=\mathbf{1}_{A_\ell}(v_\ell)+\mathbf{1}_{t+A_\ell}(v_\ell)$ because the sets are disjoint.
This completes the induction step of the proof.
\end{proof}

Following \cite[Definition 2.2]{MR1749670}, we say that a continuous map $f:X \to Y$ is $\mathcal{U}$-compatible if there exist a finite open cover $\mathcal{V}$ of $f(X)$ such that $\mathcal{U} \preceq f^{-1}(\mathcal{V})$. 

The following lemma is a   minor variation of a well known characterization of  $\dim(\mathcal{U})$ in terms of maps in to simplicial complexes,  as presented  in   \cite[Proposition $2.4$]{MR1749670}. The proof below, which we include for completeness, is nearly identical to an argument appearing in  \cite{MR1749670}. 
\begin{lemma}[Simplicial complexes and nerve maps]\label{lem:compatible_maps_complexes}
    Let $X$ is a compact metrizable space,  $\mathcal{U} \in \Cov(X)$ and $\phi:X \to \mathbb{Z}_+$ a non-negative integer valued function. Then the following are equivalent:
    \begin{enumerate}
        \item There exists an open cover $\mathcal{V}$ with $\mathcal{U} \preceq \mathcal{V}$ and 
        \[
        \ord(\mathcal{V},x) \le \phi(x) ~ \forall x \in X.
        \]
        \item There exists a simplicial complex $K$ and a $\mathcal{U}$-compatible continuous function $f:X \to K$ such that
        for every $x \in X$, the point $f(x)$ is contained in  a simplex of $K$ that has dimension at most $\phi(x)$.
    \end{enumerate}
\end{lemma}
\begin{proof}
Suppose that there exists a simplicial complex $K$  and a $\mathcal{U}$-compatible continuous function $f:X \to K$ such that
        for every $x \in X$ the point $f(x)$ belongs to a simplex of $K$ having dimension at most $\phi(x)$. Let $\mathcal{W} \in \Cov(K)$ be such that $\mathcal{U} \preceq f^{-1}(\mathcal{W})$. Find $\mathcal{W}_0 \in \Cov(X)$ such that $\mathcal{W} \preceq \mathcal{W}_0$ and $\ord(\mathcal{W}_0,v) \le k$ whenever $w$ belongs to a $k$-dimensional simplex in $K$. Then $\mathcal{V}: =f^{-1}(\mathcal{W}_0) \in \Cov(X)$ satisfies $\mathcal{U} \preceq \mathcal{V}$ and $\ord(\mathcal{V},x) \le \phi(x)$ for every $x \in X$.

    Conversely, suppose that there exists $\mathcal{V}=\{V_1,\ldots,V_r \in \Cov(X)$ with $\mathcal{U} \preceq \mathcal{V}$ and $\ord(\mathcal{V},x) \le \phi(x)$ for all $x \in X$.
    Let $f=(f_1,\ldots,f_r):X \to [0,1]^r$ be a partition of unity subordinate to the cover $\mathcal{V}$. 
    The image $f(X)$ is contained in the simplicial complex $K \subseteq [0,1]^d$ whose simplices are of the form 
    \[
    S_J := \left\{ (v_1,\ldots,v_d) \in [0,1]^d~:~ \sum_{j \in J}v_j=\sum_{j=1}^d v_j =1\right\},
    \]
    where $J \subseteq \{1,\ldots,J\}$ ranges over subsets such that $\bigcap_{j \in J}V_j \ne \emptyset$.
    Then $f:X \to K$ is $\mathcal{V}$-compatible hence $\mathcal{U}$-compatible and every $x \in X$ belongs to a simplex in $K$ having dimension at most $\phi(x)$.    
\end{proof}

The following lemma refines  of the subadditivity property
\[
\dim(\mathcal{U}_1\vee \mathcal{U}_2) \le \dim(\mathcal{U}_1)+\dim(\mathcal{U}_2),
\]
as presented in \cite{MR1749670}.  We need the  slightly more involved  statement below for later applications. Once again, we include  a proof for completeness and reassurance.

\begin{lemma}\label{lem:ord_subadditive}
Let $X$ be a compact metrizable space and $\mathcal{U}_1,\mathcal{U}_2 \in \Cov(X)$. Then there exists $\mathcal{V} \in \Cov(X)$ such that
$\mathcal{U}_1 \vee \mathcal{V}_2 \preceq \mathcal{V}$
and 
\begin{equation}
    \ord(\mathcal{V},x) \le \ord(\mathcal{U}_1,x) + \ord(\mathcal{U}_2,x) ~\forall x \in X.
\end{equation}
\end{lemma}

\begin{proof}
   By \Cref{lem:compatible_maps_complexes} there exist simplicial complexes $K_1$ and $K_2$ and maps $f_1:X \to K_1$, $f_2:X \to K_2$ that are $\mathcal{U}_1$-compatible  and $\mathcal{U}_2$-compatible respectively and so that every  $x \in X$ is mapped via $f_i$ into  a simplex of $K_i$ having dimension at most $\ord(\mathcal{U}_i,x)$ for $i=1,2$.
    Let $f=(f_1,f_2):X \to K_1\times K_2$. Then $f$ is $\mathcal{U}_1 \vee \mathcal{U}_2$ compatible. Also,  every $x \in X$  is mapped via $f$ into   a simplex of $K_1 \times K_2$ having dimension at most $\ord(\mathcal{U}_1,x)+\ord(\mathcal{U}_2,x)$. Using \Cref{lem:compatible_maps_complexes} again, there exists a an open cover $\mathcal{V} \in \Cov(X)$ with $\mathcal{U}_1 \vee\mathcal{V}_2 \preceq \mathcal{V}$ so that 
    \[
    \ord(\mathcal{V},x) \le \ord(\mathcal{U}_1,x)+\ord(\mathcal{U}_2,x) ~\forall x \in X.
    \]
\end{proof}

The following lemma is a (slight refinement)  of the  ``Ostrand-Kolmogorov method'', introduced by Ostrand  in the context of  Hilbert's problem 13  \cite{MR177391},  following earlier work of Kolmogorov \cite{MR111809} and Arnold  \cite{MR111808}. The exposition below essentially follows  \cite{levin2023finitetoone}, and again is included mostly for completeness: 
\begin{lemma}[Ostrand-Kolmogorov covers]\label{lem:ostrand_kolmogorov_covers}
Let $X$ be a compact metric space and let $\mathcal{U} \in \Cov(X)$ be a finite open cover of $X$. Then for any $k \in \N$ there exists $k$ finite families $\mathcal{C}_0,\ldots,\mathcal{C}_{k-1}$, each of which consists of closed, pairwise disjoint closed subsets of $X$ so that
    $\mathcal{C}_j \succeq \mathcal{U}$ for every $0\le j < k$
 and
 \begin{equation}\label{eq:ostrand_ord}
 \sum_{j=0}^{k-1}\overline{\mathbf{1}}_{\mathcal{C}_j}(x) \le \ord(\mathcal{U},x) ~\forall x \in X.
 \end{equation}
\end{lemma}
\begin{proof}
    Let $\mathcal{U} =\{U_1,\ldots,U_r\} \in \Cov(X)$ be a finite open cover and $k \in \mathbb{N}$. 
    
Let  
$f_1,\ldots,f_r:X \to [0,1]$ be a partition of unity subordinate to the cover  $\mathcal{U}$ and let $f:X \to [0,1]^r$ be 
given by $f(x)=(f_1(x),\ldots,f_r(x))$.

Let $F = \{ 0  < t_1< \ldots < t_N <1 \}  \subset (0,1)$ 
be a finite set of numbers so that the elements of $F\cup\{1\}$ 
are linearly independent over $\mathbb{Q}$ and so that for every $1 \le i < N$ we have that $t_{i+1} < t_i +\frac{1}{r}$.

Let 
\begin{equation}\label{eq:delta_def}
\delta := \min\left\{ |1-\sum_{j=1}^\ell s_j | ~:~ 1\le \ell \le r ,~s_1,\ldots,s_\ell \in  F  \right\}.
\end{equation}
Because the elements of $F\cup\{1\}$ are linearly independent over $\mathbb{Q}$, we have that $\delta>0$.

Choose $\epsilon >0$ so that $\epsilon < \frac{1}{2kr} \delta$ and $\epsilon < \frac{1}{2k}\min\{ |t_{i+1}-t_i| ~:~ 1 \le i< N\}$.
For every $0 \le j <k$ let 
\[\mathcal{I}_j
= \{[0,t_1+(j-1)\epsilon]\} \cup \{ [t_i +j\epsilon,t_{i+1}+(j-1)\epsilon]~:~ 1\le i < N\} \cup \{[t_N+j \epsilon,1]\}.
\]
Then each $\mathcal{I}_j$ is a collection of pairwise disjoint closed intervals of length less than $\frac{1}{r}$. 
Note that  by the choice $\epsilon < \frac{1}{2k}\min\{ |t_{i+1}-t_i| ~:~ 1 \le i< N\}$, the sets $V_j := [0,1]\setminus \bigcup_{I \in \mathcal{I}_j}I$ are pairwise disjoint. It follows that for every $t \in [0,1]$
there exists at most one index $0\le j < k$ such that $t \not \in \bigcup_{I \in \mathcal{I}_j}I$, so the following holds:
\begin{equation}\label{eq:I_ord_bound}
\sum_{j=0}^{k-1}\overline{\mathbf{1}}_{\mathcal{I}_j}(t) \in \{0,1\} ~ \forall t \in [0,1].
\end{equation}
Moreover, if $\sum_{j=0}^{k-1}\overline{\mathbf{1}}_{\mathcal{I}_j}(t) = 1$ then there exists $1\le i \le N$ such that $|t-t_i| < k \epsilon < \frac{1}{2r}\delta$.
\begin{equation}\label{eq:ord_t_close_to_t_i}
    \sum_{j=0}^{k-1}\overline{\mathbf{1}}_{\mathcal{I}_j}(t) = 1 ~ \Longrightarrow  \min_{1\le i \le N}|t-t_i| < k \epsilon < \frac{1}{2r}\delta
\end{equation}
Let $T_r = \{ (v_1,\ldots,v_r) \in [0,1]^r~:~ \sum_{\ell=1}^r v_\ell=1\}$ denote the standard simplex in $\R^r$. 
Given $v \in T_d$ let  $\mathrm{supp}(v) \subseteq \{1,\ldots,r\}$ denote the set of non-zero coordinates in $v$. 
We claim that for any $v= (v_1,\ldots,v_r) \in  T_r$ we have 
\begin{equation}\label{eq:ord_v_supp_bound}
\sum_{\ell=1}^r\sum_{j=0}^{k-1}\overline{\mathbf{1}}_{\mathcal{I}_j}(v_\ell) \le |\mathrm{supp}(v)|-1.    
\end{equation}

Indeed, since $\overline{\mathbf{1}}_{\mathcal{I}_j}(0)=0$ for every $0\le j < k$ we have that for every $(v_1,\ldots,v_r) \in  T_r$ 
\[
\sum_{\ell=1}^r\sum_{j=0}^{k-1}\overline{\mathbf{1}}_{\mathcal{I}_j}(v_\ell)= \sum_{ \ell \in \mathrm{supp}(v)}\sum_{j=0}^{k-1}\overline{\mathbf{1}}_{\mathcal{I}_j}(v_\ell).
\]
By \eqref{eq:I_ord_bound}, in order to prove \eqref{eq:ord_v_supp_bound} it suffices to show that for any $(v_1,\ldots,v_r) \in T_r$ there  exists $1\le \ell_0 \le r$ such that $\sum_{j=0}^{k-1}\overline{\mathbf{1}}_{\mathcal{I}_j}(v_\ell)=0$. 
However, if $(v_1,\ldots,v_r) \in T_r$ then $\sum_{\ell\in \mathrm{supp}(v)} v_\ell =1$ so
\[
\sum_{\ell \in \mathrm{supp}(v)} \min_{1 \le i \le N}|v_\ell -t_i| \ge \min_{1\le m \le r}\min_{s_1,\ldots,s_m \in F}\left|1- \sum_{j=1}^m s_j \right| \ge \delta.
\]
It follows that there  exists $\ell_0 \in \mathrm{supp}(v)$ so that
$\min_{1 \le i \le N}|v_{\ell_0} -t_i|  \ge \frac{\delta}{r}$, and so by \eqref{eq:ord_t_close_to_t_i} $\sum_{j=0}^{k-1}\overline{\mathbf{1}}_{\mathcal{I}_j}(v_{\ell_0}) =0$, so we have proved \eqref{eq:ord_v_supp_bound}.

For every $1\le j \le k$ let

  \[
   \mathcal{D}_j = \left\{ I_1 \times\ldots \times  I_r ~:~ I_1,\ldots I_r \in \mathcal{I}_j \right\}.
   \]

Then $\mathcal{D}_1,\ldots,\mathcal{D}_k$ are collections of pairwise disjoint closed cubes in $\mathbb{R}^r$, and for every $v = (v_1,\ldots,v_r) \in T_d$ and $0\le j < k$ we have

$\overline{\mathbf{1}}_{\mathcal{D}_j}(v) \le
\sum_{\ell=1}^r\overline{\mathbf{1}}_{\mathcal{I}_j}(v_\ell)$, so summing over $0\le j < k$ we have
\[
\sum_{j=0}^{k-1} \overline{\mathbf{1}}_{\mathcal{D}_j}(v) \le 
\sum_{j=0}^{k-1}\sum_{\ell=1}^r\overline{\mathbf{1}}_{\mathcal{I}_j}(v_\ell),
\]
So using \eqref{eq:ord_v_supp_bound} we have

\begin{equation}\label{eq:sim_D_lower_bound}
\sum_{j=0}^{k-1}\overline{\mathbf{1}}_{\mathcal{D}_j}(v) \le |\mathrm{supp}(v)|-1.
\end{equation}

Let \[\mathcal{C}_j= \{ f^{-1}(B)~:~ B \in \mathcal{D}_j \}.\]
Then $\mathcal{C}_1,\ldots,\mathcal{C}_k$ are are collections of pairwise disjoint closed subsets of $X$.
Since $\sum_{\ell=1}^r f_\ell(x) \ge 1$ for every $x \in X$, it follows that for every $x \in X$ there exists $1 \le \ell \le r$ such that $f_\ell(x) \ge \frac{1}{r}$. Since the length of each  interval $I \in \bigcup_{j=1}^r \mathcal{I}_j$ is strictly less than $\frac{1}{r}$, it follows that for every every $B = I_1\times  \ldots \times I_r \in \bigcup_{j=1}^k \mathcal{D}_j$ such that $f^{-1}(B) \ne \emptyset$ there exists $1\le j \le k$ such that $I_j \subseteq (0,1]$, hence $f^{-1}(B) \subseteq U_j$. This prove that for every $\mathcal{U} \preceq \mathcal{C}_j$ for every $0 \le j <k$.

Since  $f_1,\ldots,f_r$ are a partition of unity subordinate to $\mathcal{U}$ it follows that $|\mathrm{supp}(f(x))| \le \ord(\mathcal{U},x)+1$ for every $x \in X$, so from \Cref{eq:sim_D_lower_bound} we get that
\[
\sum_{j=0}^{k-1}\overline{\mathbf{1}}_{\mathcal{C}_j}(x) \le   \ord(\mathcal{U},x).
\]

\end{proof}

\section{Dimension in the category of topological-measure pairs}
\label{sec:TM_pairs}

Given a topological metric  
 space $X$, we denote by $\Prob(X)$ the simplex of Borel probability measures on $X$. The space $\Prob(X)$ will be equipped with the weak-$*$ topology, which makes it a compact metrizable topological space.

By a \emph{topological-measure pair} we will mean a pair $(X,P)$, where $X$ is a compact metric space and $P \subseteq \Prob(X)$ is a closed, convex subset. Topological-measure pairs have been considered the context of mean dimension in   \cite{elliott2025smallboundarypropertymathcal}.

A \emph{morphism} between topological-measure pairs $(X,P)$ and $(Y,Q)$ is a continuous function $f:X \to Y$ such that the induced affine map $f_*:\Prob(X) \to \Prob(Y)$ satisfies $f_*(P) \subseteq Q$. An isomorphism of $(X,P)$ and $(Y,S)$ is a morphism $f:X \to Y$ that is also a homeomorphism  between $X$ and $Y$ and so that $f_*(P)=Q$. Similarly, an embedding of $(X,P)$ into $(Y,Q)$ is a  morphism $f:X \to Y$ that is also an embedding of  $X$ into  $Y$.

Given a topological-measure pair $(X,P)$ with $P \ne \emptyset$ and a finite collection $\mathcal{V}$ of Borel subsets of $X$ denote:
\begin{equation}\label{eq:ord_P_def}
\ord(\mathcal{V},P) = \sup_{\mu \in P} \int_X \ord(\mathcal{V},x) d\mu(x).    
\end{equation}

For $\mathcal{U} \in \Cov(X)$ denote:
\begin{equation}\label{eq:dim_P_def}
\dim(\mathcal{U},P) = \inf_{\mathcal{V} \in \Cov(X)~:~ \mathcal{V} \succeq \mathcal{U}}\ord(\mathcal{V},P),
\end{equation}

Define the \emph{dimension} of the topological-measure pair $(X,P)$ to be 
\begin{equation}\label{eq:dim_X_T_def}
\dim(X,P) = \sup_{\mathcal{U} \in \Cov(X)}\dim(\mathcal{U},P).
\end{equation}

When  $P = \emptyset$ , we declare $\dim(X,P)=\dim(X,\emptyset)=-\infty$.
Observe that 
for any (non-empty) compact metric space $X$, $\dim(X,\Prob(X))$ is precisely the Lebesgue covering dimension of $X$.
Note that $\dim(\emptyset)=-\infty$ according to our convention.

\begin{proposition}\label{prop:dim_emb_monotone}
If a topological-measure pair $(X,P)$ embeds in another pair $(Y,Q)$, then $\dim(X,P) \le \dim(Y,Q)$. 
In particular, $\dim(X,P)$ is an invariant of isomorphism in the category of topological-measure pairs.
\end{proposition}
\begin{proof}
Suppose that $f:X \to Y$ is an embedding of $(X,P)$ into $(Y,Q)$. To prove that $\dim(X,P) \le \dim(Y,Q)$ we need to show that for any open cover $\mathcal{U} \in \Cov(X)$ there exists an open cover $\mathcal{U}_0 \in \Cov(Y)$ such that for any $\dim(\mathcal{U},P) \le \dim(\mathcal{U}_0,Q)$. 

Choose any  $\mathcal{U} \in \Cov(X)$.  Since $f:X \to Y$ is a an embedding, the image of  every $U \in \mathcal{U}$ is relatively open in $f(X)$, so there exists an open set $W_U \subseteq Y$ such that $f(U) = W_U \cap X$.  Choose one open  set $W_U \subseteq Y$ as above  for every $U \in \mathcal{U}$ and let $\mathcal{U}_0 \in \Cov(X)$ be the open cover of $Y$ that consisting of the collection $\{W_U~:~ U \in \mathcal{U}\}$  together with the open set $Y \setminus f(X)$. 
We will show that  $\dim(\mathcal{U},X) \le \dim(\mathcal{U}_0,Y)$ by showing that for any $\mathcal{V}_0 \in \Cov(Y)$ such that $\mathcal{V}_0 \succeq \mathcal{U}_0$ there exists $\mathcal{V} \in \Cov(X)$ such that $\mathcal{V} \succeq \mathcal{U}$ and $\ord(\mathcal{V},P) \le \ord(\mathcal{V}_0,Q)$. Indeed given $\mathcal{V}_0 \in \Cov(Y)$ such that  $\mathcal{V}_0 \succeq \mathcal{U}_0$ let $\mathcal{V} = \{ f^{-1}(V)~:~ V \in \mathcal{V}_0\}$. Then by construction of $\mathcal{U}_0$ we have that  $\mathcal{V} \succeq \mathcal{U}$ . Since $f:X \to Y$ is an embedding, $\ord(\mathcal{V},x) \le \ord(\mathcal{V}_0,f(x))$ for every $x \in X$, and so 
\[
\sup_{ \mu \in P}\int_X \ord(\mathcal{V},x)d\mu(x)\le \sup_{\mu \in P}\int \ord(\mathcal{V}_0,y) d(f_*\mu)(y)
\]
Since $f_* \mu \in Q$ for every $\mu \in P$ we conclude that $\ord(\mathcal{V},P) \le \ord(\mathcal{V}_0,Q)$.

\end{proof}

\section{Dynamical systems and their associated topological-measure pairs}\label{sec:dynamical_systems}
Throughout this paper by a \emph{dynamical system} we will mean a pair $(X,T)$ where $X$ is a compact metric space and $T$ is a homomorphism from a countable, discrete group $\Gamma$ into the group  $\homeo(X)$ of self-homeomorphisms of $X$. 
 In this situation we will say that $\Gamma$ is the acting group of $(X,T)$ and 
 that $(X,T)$ is a $\Gamma$-dynamical system. Given a $\Gamma$-dynamical system $(X,T)$ and $\gamma \in \Gamma$ we denote by  $T_\gamma \in \homeo(X)$ the homeomorphism corresponding to the image of $\gamma$ under $T$. 
In the category of $\Gamma$-dynamical systems, the morphisms are continuous, $\Gamma$-equivariant maps.

We will denote by $\Prob(X,T)$ the simplex of $T$-invariant Borel probability measures on $X$.

There is  a natural functor from the category of $\Gamma$-dynamical systems to the category of topological-measure pairs sending a $\Gamma$-dynamical system $(X,T)$  to the pair $(X,\Prob(X,T))$. 
In particular, the pair $(X,\Prob(X))$ is an invariant of topological conjugacy \footnote{In fact, the pair $(X,\Prob(X,T))$ is an invariant of topological orbit equivalence.}.


In the rest of the paper all the topological-measure pairs will be those that arise from a dynamical system.
Given a dynamical system $(X,T)$ and an open cover $\mathcal{U} \in \Cov(X)$,we use the following abbreviations:
\begin{itemize}
    \item $\ord(\mathcal{U},T):=\ord(\mathcal{U},\Prob(X,T))$
    \item $\dim(\mathcal{U},T):=\dim(\mathcal{U},\Prob(X,T))$
    \item $\dim(X,T):=\dim(X,\Prob(X,T))$
\end{itemize}

\section{The small boundary property and dimension zero for dynamical systems}\label{sec:SBP}
In this section we prove that a dynamical system $(X,T)$ has the small boundary property if and only if $\dim(X,T)=0$.
The small boundary property,  originates in the work of Shub and Weiss  \cite{MR1125888} that dealt with the following natural question: When can one lower the topological entropy of a dynamical system by taking (continuous) factors? The small boundary property is a dynamical analog of the topological property of being totally disconnected. 

Let $(X,T)$ be a dynamical system.
The capacity of a Borel set $A \subseteq X$ is given by
\[
\capacity(A,T) =  \sup_{\mu \in P(X,T)} \mu(A).
\]
Direct unraveling of the definitions shows that the notion of capacity of a set can be expressed using the notion of  order introduced in the previous section as follows:
\[
\capacity(A,T) = \ord(\{A\},T) +1.
\]
If $\capacity(A,T)=0$ for $A\subset X$, we say that is  \emph{$T$-small}.
Note that if $T$ is the trivial action on a compact set $X$, then 
\[
\capacity(A,T)= \begin{cases}
    1 & A \ne \emptyset\\
    0 & A = \emptyset.
\end{cases}
\]
In particular for a Borel set $A \subseteq X$ is small with respect to the trivial action if and only if it is empty.
More generally, if all the orbits of $T$ are finite then $A = \emptyset$ is the only $T$-small set.

A dynamical system $(X,T)$ has the \emph{small boundary property (SBP)} if there is a basis for the topology of $X$ that consists of sets whose boundary is $T$-small. Explicitly: $(X,T)$ has the small boundary property if for every open set $U$ and any $x \in U$ there is an open  neighborhood $V$ of $x$ such that $\capacity(\partial V,T)=0$ and $V \subseteq U$.

Elliott and Niu observed that the small boundary property is actually associated to a topological measure pair \cite{elliott2025smallboundarypropertymathcal}.  

Given a set $K \subseteq X$ and $\delta >0$ let
\[
B_\delta(K)= \left\{ x \in X~:~ \inf_{y \in K}\rho(x,y) < \delta \right\}.
\]

The following  result about the capacity  appears in \cite[Lemma 6.3]{MR1793417}:
\begin{lemma}\label{lem:cap_continuous}
    Let $(X,T)$ be a dynamical system and let $K \subseteq X$ be a closed subset. Then
    \[
    \capacity(K,T)= \lim_{\delta \to 0}\capacity(B_\delta(K),T).
    \]
\end{lemma}
\begin{proof}
    Since $A \mapsto \capacity(A,T)$ is monotone non-decreasing,
    any Borel set $K \subseteq X$ we have 
    \[
    \capacity(K,T) \le \lim_{\delta \to 0}\capacity(B_\delta(K),T),
    \]
    and the limit on the right hand side exists due to monotonicity.
     To show reverse inequality, suppose that
    \[
    \lim_{\delta \to 0}\capacity(B_\delta(K),T) > t.
    \]
    So for every $n \in \N$ we can find a $\mu_n \in \Prob(X,T)$ so that $\mu_n(B_{\frac{1}{n}}(K))>t$. By compactness of $\Prob(X,T)$, after possibly replacing $(\mu_n)_{n=1}^\infty$ by a subsequence, we can assume that $(\mu_n)_{n=1}^\infty$ converges to $\mu \in \Prob(X,T)$.
    Let $g_n:X \to [0,1]$ ,$n \in \N$ be a sequence of continuous functions so that $(g_n(x))_{n \in \N}$ is monotone decreasing for every $x \in X$ and so that $\lim_{n\to \infty} g_n(x) =1_K(x)$.
    For any $\epsilon >0$ there exists $n \in \N$ such that $g_n(x) \ge 1-\epsilon$  for all $x \in B_{\frac{1}{n}}(K)$. It follows that $\int g_n(x) d\mu_m(x) \ge t(1-\epsilon)$ for all $m \ge n$.
    Hence $\int g_n(x) d\mu \ge t(1-\epsilon)$ for all $n \in \N$. By monotone convergence $\mu(A) = \int 1_A(x) d\mu(x) \ge t(1-\epsilon)$. Since this holds for any $\epsilon>0$, we conclude that $\mu(A) \ge t$, so $\capacity(A,T) \ge t$.
    This completes the proof of the reverse inequality.
\end{proof}

\begin{lemma}\label{lem:dim_zero_small_boundary_cover}
Suppose that $\dim(X,T)=0$ then  for any $\epsilon>0$ there exists a finite open cover $\mathcal{U}$ of $X$ with $\mesh(\mathcal{U})< \epsilon$ and $\capacity(\bigcup_{U \in \mathcal{U}} \partial U,T) < \epsilon$.
\end{lemma}
\begin{proof}
    Let $\epsilon>0$ be given. Because $\dim(X,T)=0$ for every $\epsilon >0$ there exists an open cover $\mathcal{U}'=\{U'_1,\ldots,U'_\ell\}$ with $\mesh(\mathcal{U}')< \epsilon$ such that 
\[
\sup_{\mu \in \Prob(X,T)}\int_X \ord(\mathcal{U}',x) d\mu(x) < \epsilon ~\forall x \in X.
\]
Let 
\[B=\left\{ x \in X~:~ \ord(\mathcal{U}',x) \ge 1\right\}  \]
It follows that 
$\capacity\left( B,T\right) < \epsilon$.
By \Cref{lem:smaller_cover_munkres} we can find open sets $U_1,\ldots, U_\ell$ so that $\overline{U_j} \subseteq U'_j$ and so that $\mathcal{U}=\left\{U_1,\ldots,U_\ell\right\}$ is still a cover of $X$. 
Suppose that $x \in \partial U_j$  for some $1 \le j \le \ell$. Then $x \not \in U_j$, because $U_j$ is open, but $x \in U'_j$. Since $\{U_1,\ldots,U_\ell\}$ is a cover, there exist $k \in \{1,\ldots,\ell\} \setminus \{j\}$ such that $x \in U_k$. So $x \in U'_j \cap U'_k$ so $\ord(\mathcal{U}',x) \ge 1$.
We conclude that  $\bigcup_{j=1}^\ell \partial U_j \subseteq B$. 
It follows that for any $\mu \in \Prob(X)$ 

\[
\capacity(\bigcup_{j=1}^\ell \partial U_j ) \le \int_{\bigcup_{j=1}^\ell \partial U_j  } \ord(\mathcal{U}',x) d\mu(x) < \epsilon.
\]

This proves that  indeed $\capacity(\bigcup_{U \in \mathcal{U}} \partial U,T) < \epsilon$.

\end{proof}

\begin{theorem}\label{thm:SBP_iff_dim_0}
    A dynamical system $(X,T)$ has the small boundary property if and only if $\dim(X,T)=0$.
\end{theorem}
\begin{proof}

Suppose that $(X,T)$ has the small boundary property.
We need to prove that for every $\epsilon>0$ there exists an open cover $\mathcal{U}$ such that $\dim(\mathcal{U},T) < \epsilon$ and $\mesh(\mathcal{U}) < \epsilon$.
Let $\mathcal{U}'=\left\{U'_1,\ldots,U'_\ell\right\}$ be a finite open cover of $X$ with $\mesh(\mathcal{U}')< \epsilon/2$ by sets with zero capacity boundaries. Let $U_j'' = U'_j \setminus \bigcup_{ i < j} U'_j$, then the collection of sets  $\{U''_1,\ldots,U''_\ell\}$ are pairwise disjoint and cover $X$. Since the boundary of each $U''_j$ is contained in the union of the boundaries of the sets $U'_1,\ldots,U''_j$ we have that $\capacity (\partial U''_j) =0$ for every $1\le j \le \ell$. By \Cref{lem:cap_continuous}, for sufficiently small $\delta>0$ we have that 
\[
\capacity( \bigcup_{j=1}^\ell B_\delta(\partial U''_{j}),T) < \epsilon/\ell.
\]
Choose $0< \delta < \epsilon/2$ as above and let
\[K = \bigcup_{j=1}^\ell B_{\delta}(\partial U''_{j})).\]
For $1 \le j \le \ell$ let $U_j = U_j'' \cup B_{\delta}(\partial U''_j)$, and let $\mathcal{U}=\{U_1,\ldots,U_\ell\}$. 
It is clear that $\mesh(\mathcal{U}) \le \epsilon/2 + \delta < \epsilon$.

Since $\{U''_1,\ldots,U''_\ell\}$ are pairwise disjoint it follows that  if $x \in X$ and $\ord(\mathcal{U},x) >0$ then $x \in K$.
On the other hand, for any $x \in X$ we have $\ord(\mathcal{U},x)\le \ell$.
It follows that for any $\mu \in \Prob(X,T)$ we have
\[
\int_X \ord(\mathcal{U},x) d\mu(x)= \int_K \ord(\mathcal{U},x)d\mu(x) \le \ell \capacity(K,T) < \epsilon.
\]
We conclude that $\ord(\mathcal{U},T) < \epsilon$.

Now suppose that $\dim(X,T)=0$.
Let $V$ be an open set and  $x \in V$. We will show that there exists an open set $W \subseteq V$ with $x \in U$ such that $\capacity(\partial W,T)=0$.
To do this, we will construct by induction a sequence of positive numbers $(\epsilon_n)_{n=1}^\infty$ such that $\lim_{n \to \infty} \epsilon_n =0$ and a sequence of open sets 
\[x \in W_0 \subseteq W_1 \subseteq \ldots \subseteq W_n \subseteq \ldots\]
So that for every $n \in \mathbb{N}$ the following hold:
\begin{enumerate}
    \item $\overline{W_{n-1}} \subseteq W_n \subseteq B_{\epsilon_{n-1}}(\overline{W_{n-1}})$.
    \item $B_{\epsilon_n}(W_n) \subseteq V$
    \item $\capacity(B_{\epsilon_n}(\partial W_n),T) < \frac{1}{n}$.
    \item $B_{\epsilon_n}(\partial W_n) \subseteq B_{\epsilon_{n-1}}(\partial W_{n-1})$.
\end{enumerate}
Then $W = \bigcup_{n=1}^\infty W_n$ will be an open set such that $ x \in W$ and $W \subseteq V$ and $\capacity(W,T)=0$.
To start the induction, let $W_0$ be an open set with $x \in X$ so that $\overline{W_0} \subseteq V$ and choose $\epsilon_0>0$ smaller than the distance between $\overline{W_0}$ and $X \setminus V$.
Now suppose that $W_0,\ldots,W_{n-1}$ and $\epsilon_0,\ldots,\epsilon_{n-1}$ have been constructed.

By \Cref{lem:dim_zero_small_boundary_cover} we can find  an open cover $\mathcal{U}_n$  of $X$ such that $\mesh(\mathcal{U}_n) < \epsilon_{n-1}$ and $\capacity(\bigcup_{U \in \mathcal{U}_n,T} \partial U) < \frac{1}{n}$. Thus there exists $\epsilon_n>0$ sufficiently  small so that $\capacity(B_{\epsilon_n}(\bigcup_{U \in \mathcal{U}_n} \partial U),T) < \frac{1}{n}$. Furthermore, choose $\epsilon_n< \epsilon_{n-1}$. For every $x \in \partial W_{n-1}$ choose an element $U_x \in \mathcal{U}_n$  that contains $x$. Let $W_n= W_{n-1} \cup \bigcup_{x \in \partial W_{n-1}} U_x$. Then 
$ \partial W_{n-1} \subseteq \bigcup_{U \in \mathcal{U}_n}\partial U$, so $\capacity(B_{\epsilon_n}(\partial W_n),T) < \frac{1}{n}$. This completes the induction step, and so the proof of \Cref{thm:SBP_iff_dim_0} is complete.
\end{proof}

Lindenstrauss and Weiss proved that for any action of an amenable group the small boundary property implies zero mean dimension \cite[Theorem 5.4]{MR1749670}. As we prove in \Cref{thm:mdim_le_dim}, the mean dimension of $(X,T)$ is always bounded from above by $\dim(X,T)$.
It was 
shown by  Lindenstrauss that 
a $\mathbb{Z}$-dynamical system $(X,T)$ that admits a free minimal factor
has zero mean dimension if and only if it has the small boundary property \cite[Theorem 6.2]{MR1793417}. As we shall see in \Cref{sec:dim_amenable}, for such $(X,T)$  we always have $\mdim(X,T)=\dim(X,T)$. So  \Cref{thm:SBP_iff_dim_0} can be viewed as a refinement of  \cite[Theorem 6.2]{MR1793417} and \cite[Theorem 5.4]{MR1749670}.

\section{Separating closed subsets and a Baire category lemma}
\label{sec:Baire_Category}

 The  Baire category approach has been used repeatedly in classical dimension theory, and has also been instrumental in the  proving fundamental results about mean dimension. 
 In this section we formulate a simple auxiliary lemma that encapsulates a Baire category argument  used later in this paper.
 The statement of the lemma and its short  proof are presented  marely for completeness as the argument is quite standard. We do not claim novelty for the approach presented in this section.

Let $X$ be a compact metric space and $\mathcal{C} \subseteq 2^X$ a collection of pairwise disjoint subsets of $X$. We say that a function $f:X \to Y$ \emph{separates} $\mathcal{C}$ if $f(C_1) \cap f(C_2)=\emptyset$ for  every pair of distinct elements $C_1,C_2 \in \mathcal{C}$.

\begin{lemma}\label{lem:separate}
    Let $X$ be a compact metric space and $d \in \mathbb{N}$. For every $1\le j \le d$ and $n \in\ \N$ let  $\mathcal{C}_{j,n}$ be a collection of  pairwise disjoint closed subsets of $X$ so that  $\lim_{n \to \infty}\mesh(\mathcal{C}_{j,n})=0$ for every $1\le j \le d$.
    For every $n \in \mathbb{N}$ let
    \[
    G _n = \left\{ f = (f_1,\ldots,f_d) \in C(X,[0,1]^d)~:~ f_j \mbox{ separates } \mathcal{C}_{j,n} ~\forall 1\le j \le d  \right\}
    \]
    Then the set 
    \[
    G= \bigcap_{N=1}^\infty\bigcup_{n=N}^\infty G_n
    \]
    is a dense $G_\delta$ subset of $C(X,[0,1]^d)$.
    Moreover, let $A \subset [0,1]$ be a finite set and for every $n \in \N$ let 
    \[
    \tilde G_n = \left\{ 
    (f_1,\ldots,f_d) \in G_n ~:~ f_j(x) \not \in A ~\forall x \in \bigcup_{C \in \mathcal{C}_{j,n}}C,~ 1\le j \le d 
    \right\}.
    \]
    Then the set 
    \[
    \tilde G= \bigcap_{N=1}^\infty\bigcup_{n=N}^\infty \tilde G_n
    \]
    is also a dense $G_\delta$ subset of $C(X,[0,1]^d)$.
\end{lemma}
\begin{proof}
    Since $C(X,[0,1]^d)$ is a complete metric space, by the Baire category theorem  a countable intersection of dense open sets is also dense, so it suffices  to show that for every $N \in \mathbb{N}$ the set 
    \[
    D_N = \bigcup_{n=N}^\infty G_n
    \]
    is dense and open.
    Because   $\mathcal{C}_{j,n}$ is a finite collection of pairwise disjoint closed (hence compact) pairwise disjoint subsets of $X$, it follows that for any  $f \in C(X,[0,1])$ that separates $\mathcal{C}_{j,n}$ there exist $\epsilon>0$ such that $|f(x_1)-f(x_2)| > \epsilon$ for any $x_1,x_2 \in X$ that belong to distinct elements of $\mathcal{C}_{j,n}$. So any $\|g-f\|_\infty  < \frac{\epsilon}{2}$ also separates $\mathcal{C}_{j,n}$. This shows that set $G_n$ is open for every $n \in \N$, hence $D_N$ is also open for every $N \in \N$. 
    Let us check  that $D_N$ is dense in $C(X,[0,1]^d)$. Let $f =(f_1,\ldots,f_d) \in C(X,[0,1]^d)$ be an arbitrary continuous function and $\epsilon >0$. Since $X$ is compact there exists $\delta >0$ such that $f(C)$ has diameter less then $\epsilon/2$ for any set $C \subseteq X$  whose diameter is  less than $\delta>0$.  Since $\lim_{n \to \infty}\mesh(\mathcal{C}_{j,n})=0$ we can find $n \in \mathbb{N}$ such that $\mesh(\mathcal{C}_{j,n})< \delta/2$ for every $1
    \le j \le d$. For every $C \in \mathcal{C}_{j,n}$ choose $x_C \in C$ and $t_c \in [0,1]$ such that $|t_c-f_j(x_c)|< \delta/2$. Define a continuous function 
    \[
    \tilde g_j: \bigcup_{C \in \mathcal{C}_{j,n
}}C \to [0,1]
    \]
by
\[
\tilde g_j (x) = \sum_{C \in \mathcal{C}_{j,n}}t_C 1_{C}(x).
\]
Then because $\mesh(\mathcal{C}_{j,n}) < \delta$ it follows that 
\[
|\tilde g_j(x)- f_j(x)| < \epsilon ~ \forall x \in \bigcup_{C \in \mathcal{C}_{j,n}}C.
\]
By Tietze extension theorem we can extend $\tilde g$ to a continuous function $g \in C(X,[0,1]^d)$ such that $\|g_j-f_j\|_\infty < \epsilon$. Because the values $t_C$ are distinct, $g_j$ separates $\mathcal{C}_{j,n}$, hence $g=(g_1,\ldots,g_d) \in D_N$. 
This completes the proof that $D_N$ is a dense subset of $C(X,[0,1]^d)$, therefore the proof that $G$ is a dense $G_\delta$ subset of $C(X,[0,1]^d)$ is complete.
To prove that $\tilde G$ is also a dense $G_\delta$ subset, it suffices to show that for every $N \in \mathbb{N}$ set $\tilde D_N = \bigcup_{n=N}^\infty \tilde G_n$ is  dense and open. The only required modification is that we need to make sure that the elements $t_c \in [0,1]$ satisfying $|t_c f_j(x_c)| < \delta/2$ are not in the given finite set $A$. This is clearly possible.
\end{proof}

\section{Dimension for free actions of finite groups}\label{Sec:finite_groups}

In this section we compute the dimension of dynamical systems over a \emph{finite} group. To simplify the discussion and the statement of the result, in this section only we restrict to \emph{free} actions. 
Recall that a $\Gamma$-dynamical system is called \emph{free} if $T_\gamma(x) \ne T_{\gamma'}(x)$ for every $x \in X$ and every distinct $\gamma,\gamma' \in \Gamma$.

\begin{theorem}\label{thm:dim_finite_group_action}
Let $\Gamma$ be a finite group and let $(X,T)$ be a free $\Gamma$-dynamical system.
Then \[\dim(X,T)= \frac{1}{|\Gamma|}\dim(X).\]  
\end{theorem}
For $\Gamma$-dynamical system $(X,T)$ with $\Gamma$ a finite group the definition of mean dimension directly implies that $\mdim(X,T) = \frac{1}{|\Gamma|}\dim(X)$. \Cref{thm:dim_finite_group_action} 
thus establishes that for free actions of finite groups, our new invariant coincides with mean dimension: 
\[\dim(X,T)=\mdim(X,T).\]
Before proving \Cref{thm:dim_finite_group_action} we will prove an auxiliary lemma that essentially says you can choose a function that "misses" a set of values most of the time, where ``most of the time'' is quantitatively bounded by the dimension of the space:
\begin{lemma}\label{lem:avoid_finite_values_dim}
Let  $\Gamma$ be a finite group and $(X,T)$ a free $\Gamma$-system with $\dim(X)<\infty$. Let $k \in \N$ and let $A \subset [0,1]$ be a finite set. Then there is a dense and open set of continuous functions $f:X \to [0,1]^k$ with the property that 
\begin{equation}
    \sup_{x \in X}\sum_{j=1}^k\sum_{\gamma \in \Gamma}\mathbf{1}_{A}(f(T_\gamma(x))_j) \le \dim(X).
\end{equation}

\end{lemma}
\begin{proof}
Let $d= \dim(X)$ and let $A \subset [0,1]$ be a finite set.
Since $(X,T)$ is a free $\Gamma$-action and $\Gamma$ is a finite group, for every $x\in X$ there is an open neighborhood $U_x$ such that $(T_\gamma(\overline{U_x}))_{\gamma \in \Gamma}$ are pairwise disjoint. By compactness of $X$, there exist a finite open cover $\mathcal{U}$ of $X$ such that  $(T_\gamma(\overline{U}))_{\gamma \in \Gamma}$ are pairwise disjoint for every $U \in \mathcal{U}$.

It suffices to prove that for every $U \in \mathcal{U}$ there is a dense $G_\delta$ set of functions $f \in C(X,[0,1]^k)$ satisfying
\begin{equation}\label{eq:sup_x_in_U_le_d}
   \sup_{x \in \overline{U}}\sum_{j=1}^k\sum_{\gamma \in \Gamma}\mathbf{1}_{A}(f(T_\gamma(x))_j) \le d.
\end{equation}

Fix $U \in \mathcal{U}$.
Because $\Gamma$ is finite, using the fact that $\dim(X)=d$, for every $n \in \N$ we can find an open cover $\mathcal{V}_n \in \Cov(X)$ such that $\mesh(T_\gamma(\mathcal{V}_n)) < \frac{1}{n}$ for every $\gamma \in \Gamma$. Using \Cref{lem:ostrand_kolmogorov_covers} for every $n \in \N$ we can find collections 
$(\mathcal{C}_{n,j,\gamma})_{1\le j \le k, \gamma \in \Gamma}$ such that each $\mathcal{C}_{n,j,\gamma}$ consists of pairwise disjoint closed subsets of $X$, and so that for every $1\le j \le k$ and $\gamma \in \Gamma$ we have $\mathcal{C}_{n,j,\gamma} \succeq \mathcal{V}_n$   and
\begin{equation}\label{eq:Ostrand_cover_inequality}
\sum_{j=1}^k \sum_{\gamma \in \Gamma} \overline{\mathbf{1}}_{\mathcal{C}_{n,j,\gamma}}(x) \le d.
\end{equation}

For every $n \in \N$ and $1\le j \le k$ let
\[
\mathcal{C}_{n,j} = \left\{ T_\gamma(C\cap \overline{U}) ~:~ \gamma \in \Gamma,~ C \in \mathcal{C}_{n,j,\gamma} \right\}.
\]
Then each   is a finite collection of pairwise disjoint closed sets each having diameter at most $\frac{1}{n}$.

By \Cref{lem:separate} there exists a dense $G_\delta$ set of functions $f=(f_1,\ldots,f_k) \in C(X,[0,1]^k)$ such that for infinitely many $n$'s we have that 
$f_j(x) \not\in A$ for every $x \in \bigcup_{j=1}^k\bigcup_{C \in \mathcal{C}_{n,j}}C$.

Let $f \in C(X,[0,1]^k)$ be as above.
Then there are infinitely many $n$'s such that for any   $x \in \overline{U}$ we have that $f_j(T_\gamma(x)) \not \in A$ whenever $\overline{\mathbf{1}}_{ \mathcal{C}_{n,j,\gamma}}(x)=0$.
It follows that for any $x \in \overline{U}$
\[
\sum_{j=1}^k\sum_{\gamma \in \Gamma}1_{A}(f(T_\gamma(x))_j) \le \sum_{j=1}^k\sum_{\gamma \in \Gamma}\overline{\mathbf{1}}_{\mathcal{C}_{n,j,\gamma}}(x)\le d.
\]
We have thus proved that for a dense $G_\delta$ set of function \Cref{eq:sup_x_in_U_le_d} holds.
\end{proof}

\begin{proof}[Proof of \Cref{thm:dim_finite_group_action}]
Denote $d=\dim(X)$, $k=|\Gamma|$.

We first prove that $\dim(X,T) \ge \frac{d}{k}$.
Let  $\mathcal{U} \in \Cov(X)$ be an open cover with $\dim(\mathcal{U}) =d$.
We will show that $\dim(\mathcal{U},T) \ge \frac{d}{k}$.
Indeed, choose any $\mathcal{V} \in \Cov(X)$ such that $\mathcal{V} \succeq \mathcal{U}$ then $\dim(\mathcal{U}) =d$ implies that there exists $x_0 \in X$ with $\ord(\mathcal{V},x_0) \ge d$. Then $\mu = \frac{1}{k}\sum_{\gamma \in \Gamma} \delta_{T_\gamma(x_0)} \in \Prob(X,T)$, where for $x \in X$ we denote by $\delta_x$ the Dirac measure at $x$, which by definition is the unique probability measure supported on the singleton $\{x\}$.
We have
\[
\int_X \ord(\mathcal{V},x) d\mu(x)= \frac{1}{k}\sum_{\gamma \in \Gamma} \ord(\mathcal{V},T_\gamma(x_0)) \ge \frac{1}{k}\ord(\mathcal{V},x_0) = \frac{d}{k}.
\]
This shows that 
$\ord(\mathcal{V},T) \ge \frac{d}{k}$ for any $\mathcal{V} \in \Cov(X)$ such that $\mathcal{V} \succeq \mathcal{U}$, hence $\dim(\mathcal{U},T) \ge \frac{d}{k}$. 

To prove the reverse inequality $\dim(X,T) \le \frac{d}{k}$, it suffices to show that for any $\epsilon >0$ we can find   $\mathcal{U} \in \overline{\Cov}(X)$ with $\mesh(\mathcal{U}) < \epsilon$  $\ord(\mathcal{U},T) \le \frac{d}{k}$.

 Since $\Gamma$ is a finite group, any ergodic $\mu \in \Prob(X,T)$ is of the form $\frac{1}{k}\sum_{\gamma \in \Gamma}\delta_{T_\gamma x}$ for some $x \in X$. So the condition $\ord(\mathcal{U},T) \le \frac{d}{k}$ can be rewritten as follows:
\begin{equation}\label{eq:ord_le_d_T_finite}
   \forall x \in X~ \sum_{\gamma \in \Gamma}\ord(\mathcal{U},T_\gamma(x)) \le d.
\end{equation}
By a standard compactness argument, it would be enough to show that for any $\epsilon >0$ there exists a finite cover $\mathcal{U}$ of $X$ by \emph{closed} sets with $\mesh(\mathcal{U})<\epsilon$ so that \Cref{eq:ord_le_d_T_finite} holds. 
 By \Cref{lem:cube_cover_R_d} for every $n \in \N$ there exists a cover
$\mathcal{C}_{n}$
of $[0,1]^{2d+1}$ 
by closed sets with $\mesh(\mathcal{C}_{n})< \frac{1}{n}$ and finite pairwise disjoint sets $A_{1,n},\ldots, A_{2d+1,n} \subset [0,1]$ such that for every $v=(v_1,\ldots,v_{2d+1}) \in [0,1]^{2r+1}$
\begin{equation}\label{eq:ord_C_n_upp_bound}
    \ord(\mathcal{C}_n,v) \le \sum_{\ell=1}^{2d+1}1_{A_{\ell,n}}(v_\ell).
\end{equation}

 Since $\dim(X)=d$, the classic Menger-N\"{o}beling-Pontryagin theorem implies that a dense $G_\delta$ set of continuous functions $g:X \to [0,1]^{2d+1}$ are injective.
 Applying  \Cref{lem:avoid_finite_values_dim} with $k=2d+1$, for each $n \in \N$ there exists a dense $G_\delta$ functions $g \in C(X,[0,1]^{2d+1})$ such that for such that  
 \begin{equation}\label{eq:sum_A_ell_n_uppperbound}
      \sup_{x \in X}\sum_{\ell=1}^{2d+1}\sum_{\gamma \in \Gamma}1_{A_\ell,n}(g(x)) \le d.
 \end{equation}
 By the  Baire Category theorem a countable intersection of dense $G_\delta$-sets is also dense, so there exists a embedding $g:X \to  [0,1]^{2d+1}$ that satisfies \eqref{eq:sum_A_ell_n_uppperbound} for all $n \in \N$.
 Since $\lim_{n \to \infty}\mesh(\mathcal{C}_{n}) =0$ and $g:X \to  [0,1]^{2d+1}$ is an embedding, there exist $n \in \N$ such that $\mesh(g^{-1}(\mathcal{C}_n)) < \epsilon$. Let $\mathcal{U}=g^{-1}(\mathcal{C}_n)$ for $n \in \N$ as above. 
 Combining \Cref{eq:ord_C_n_upp_bound} and \Cref{eq:sum_A_ell_n_uppperbound} we see that
 \[
 \sup_{x \in X}\sum_{\gamma \in \Gamma}\ord(\mathcal{U},x) \le d=\dim(X).
 \]

\end{proof}

A slight modification of the proof gives that $\dim(X,T) = \frac{1}{k}\dim(X)$ for every dynamical system $(X,T)$ such that size of every orbit has size $k$.
We note that in general, if every orbit of $(X,T)$ is finite then 
\[
\dim(X,T) = \sup_{k \in \mathbb{N}}\frac{1}{k}\dim(X_k),
\]
where $X_k$ is the set of point in $X$ whose orbit has size at most $k$.

\section{Dimension of cubical shifts}\label{sec:dim_cube_shift}

Given a compact metrizable space $K$, and a countable group $\Gamma$, the space $K^\Gamma$ of $K$-valued functions on $\Gamma$ is also a compact metrizable space, where $K^\Gamma$ is equipped with the product topology.  We denote by $(K^\Gamma,\shift)$ the dynamical system corresponding to the left-shift action of $\Gamma$ on $K^\Gamma$, defined by $(\shift_\gamma(x))_{\gamma'}=x_{\gamma^{-1}\gamma'}$.

\begin{theorem}\label{thm:dim_cubical_shift}
Let $\Gamma$ be a countable group. Then
for any $d \in \N$ we have
\[\dim([0,1]^d)^\Gamma,\shift) = d.
\]
\end{theorem}

The inequality $\dim([0,1]^d)^\Gamma,\shift) \ge d$ is a consequence of the fact that the set of fixed points for the shift is homeomorphic to $[0,1]^d$, and  $\dim([0,1]^d)=d$.
So \Cref{thm:dim_cubical_shift} reduces to the inequality $\dim([0,1]^d)^\Gamma,\shift) \le d$.

\begin{lemma}\label{lem:dim_U_cominatorial_upper_bound}
    Let $\Gamma$ be a countable group and 
    let $(X,T)$ be a $\Gamma$-system. Let  $\mathcal{U} \in \Cov(X)$ be an open cover, let $F  \Subset \Gamma$   be a finite set and fix $d \in \mathbb{N}$.
    Suppose  that for every $1\le j \le d$ and $\gamma \in F$ there exist a closed set $K_{j,\gamma} \subseteq X$ so that:
    \begin{enumerate}
        \item For every $x \in X$ $\ord(\mathcal{U},x) \le \sum_{j=1}^d \sum_{\gamma \in F}\mathbf{1}_{K_{j,\gamma}}(x)$
        \item For every $1\le j \le d$ the sets $(T_{\gamma}(K_{j,\gamma}))_{\gamma \in F}$ are pairwise disjoint
    \end{enumerate}
    
    Then 
    
    \[
    \ord(\mathcal{U},T) \le d.
    \]
\end{lemma}
\begin{proof}
Let $F \Subset \Gamma$, $\mathcal{U} \in \Cov(X)$ and $(K_{j,\gamma})_{1\le j \le d, \gamma \in F}$ be as in the statement.
We need to prove that for every $\mu \in \Prob(X,T)$
\[
\int \ord(\mathcal{U},x) d\mu(X) \le d.
\]
Let $\mu \in \Prob(X,T)$ be a $T$-invariant probability measure.
Since $\ord(\mathcal{U},x) \le \sum_{j=1}^d \sum_{\gamma \in F} 1_{K_{j,\gamma}}(x)$ for every $x \in X$, it suffices to prove that
\[
\int \sum_{j=1}^d \sum_{\gamma \in F} \mathbf{1}_{K_{j,\gamma}}(x) d\mu(x) \le d.
\]
Note that 
\[
\int \sum_{j=1}^d \sum_{\gamma \in F} \mathbf{1}_{K_{j,\gamma}}(x) d\mu(x) = \sum_{j=1}^d \sum_{\gamma \in F}\mu(K_{j,\gamma}).
\]
Because $\mu$ is $T$-invariant, for every $1\le j \le d$ and $\gamma \in F$ we have that $\mu(K_{j,\gamma})=\mu(T_\gamma(K_{j,\gamma}))$.
Because the sets $(T_{\gamma}(K_{j,\gamma}))$ are pairwise disjoint for every $1\le j \le d$, we have that
\[
\sum_{\gamma \in F}\mu(T_\gamma(K_{j,\gamma})) \le 1.
\]
It follows that
\[
\sum_{\gamma \in F}\mu(K_{j,\gamma}) \le 1.
\]
Summing over $1\le j \le d$ we conclude that 
\[
\int  \sum_{j=1}^d \sum_{\gamma \in F} 1_{K_{j,\gamma}}(x) \le d,
\]
concluding the proof.

\end{proof}

\begin{proof}[Proof of \Cref{thm:dim_cubical_shift}]

    We need to prove the inequality $\dim([0,1]^d)^\Gamma,\shift) \le d$. 
    It suffices to show that  for every $F \Subset \Gamma$ and $\epsilon >0$ there exists a finite cover $\mathcal{U}$ of $([0,1]^d)^\Gamma$ by closed sets such that $\mesh(\pi_F(\mathcal{U})) < \epsilon$, where $\pi_F:([0,1]^d)^\Gamma\to ([0,1]^d)^F$ is the obvious continuous surjection.
    So choose any finite set $F \Subset \Gamma$ and $\epsilon>0$.
    By \Cref{lem:cube_cover_R_d} there exists a cover $\mathcal{U}'$ of $([0,1]^d)^F$  by closed sets such that $\mesh(\mathcal{U}')<\epsilon$ and a family $(A_{\ell,\gamma})_{1\le \ell \le d,\gamma \in F}$ of pairwise disjoint  finite subsets of $[0,1]$  such that
    for every $v \in ([0,1]^d)^F$
    \[
     \ord(\mathcal{U}',x) \le \sum_{\ell=1}^d\sum_{\gamma \in F}\mathbf{1}_{A_{\ell,\gamma}}(v_{\ell,\gamma}).
    \]
    Let $\mathcal{U}$ be  pullback of $\mathcal{U}'$ via $\pi_F$.

    For $1\le j \le d$ and $\gamma \in F$, let $K_{j,\gamma} \subset ([0,1]^d)^\Gamma$ denote the set of points $(x_{j,\gamma})_{1\le j \le d,~ \gamma \in \Gamma}   \in ([0,1]^d)^\Gamma$ such that $x_{j,\gamma^{-1}} \in A_{j,\gamma}$. 
    Then the cover $\mathcal{U}$ of $([0,1]^d)^\Gamma$ and the sets $(K_{j,\gamma})_{1\le j \le d,\gamma \in F}$ satisfy the hypothesis of \Cref{lem:dim_U_cominatorial_upper_bound}.
    It follows by \Cref{lem:dim_U_cominatorial_upper_bound} that 
    \[
    \ord(\mathcal{U},\shift) \le d.
    \]
\end{proof}

\begin{corollary}\label{cor:cube_emb_dim}
    Let $\Gamma$ be any countable or finite group and let $(X,T)$ be a $\Gamma$-dynamical system. If $(X,T)$ embeds in $(([0,1]^d)^\Gamma,\shift)$ then $\dim(X,T) \le d$.
\end{corollary}

\section{Dimension for actions of amenable groups}\label{sec:dim_amenable}
In this section we discuss the quantity $\dim(X,T)$ where $(X,T)$ is a $\Gamma$-dynamical system and $\Gamma$ is a countable amenable group. 
The basic definition of mean  dimension of a dynamical system $(X,T)$  requires amenability of the acting group $\Gamma$. A more involved definition of a generalization called  ``sofic-mean dimension''  has been defined for actions of sofic groups \cite{MR3077882}. Related results about $\dim(X,T)$ and sofic mean dimension for actions of  sofic groups will be presented elsewhere.
One of the definitions of amenability of $\Gamma$ is the condition that $\Prob(X,T) \ne \emptyset$ for any $\Gamma$-dynamical system $(X,T)$.  So in this setting we always have $\dim(X,T)  \ge 0$.
We will show that  $\mdim(X,T) \le \dim(X,T)$, and that equality holds if $(X,T)$ satisfies the \emph{uniform Rokhlin property (URP)}, a property that we will soon recall.  

Let $\Gamma$ be a countable amenable group, $(X,T)$ a $\Gamma$-dynamical system and $0<t < +\infty$.
Recall that $\mdim(X,T) <t$ if and only if for every $\epsilon$ there exists a finite set $F \Subset \Gamma$ and an open  cover $\mathcal{U} \in \Cov(X)$ so that $\mesh(T_\gamma^{-1}(X))< \epsilon$ for all $\gamma \in F$ and so that $\ord(\mathcal{U}) < |F|t$.

We would like to present an equivalent formulation of $\dim(X,T)$ that does not explicitly involve taking integrals with respect to  invariant probability measures in the amenable  setting. For this purpose we recall the following  basic result relating integrals over invariant measures and averages along orbit segments.  

\begin{proposition}\label{prop:upper_semi_erg_max_value}
     Let $\Gamma$ be a finite or countable amenable group, let $(F_n)_{n=1}^\infty$ be a F\o lner sequence in $\Gamma$ and let $(X,T)$ be a $\Gamma$-dynamical system. Then for any upper semicontinuous function $f:X \to [0,+\infty)$ the following holds:
     \begin{equation}\label{eq:erg_max_avg_lim}
         \sup_{\mu \in \Prob(X,T)}\int_X f(x)d\mu(x) = 
         \lim_{n\to \infty}\frac{1}{|F_n|} \sup_{x \in X}\sum_{\gamma \in F_n}f(T\gamma(x)).
     \end{equation}
     and
     \begin{equation}\label{eq:erg_max_avg_inf}
          \sup_{\mu \in \Prob(X,T)}\int_X f(x)d\mu(x) =
          \inf_{F \Subset \Gamma}
          \frac{1}{|F|}\sup_{x \in X}\sum_{\gamma \in F}f(T\gamma(x)).
     \end{equation}
\end{proposition}
In the literature, the quantity that appears in both sides of \Cref{eq:erg_max_avg_lim} and \Cref{eq:erg_max_avg_inf} is sometimes called the \emph{maximal ergodic average of the function $f$}. The study of problems related to the maximal ergodic averages  relates to a field in dynamics known as ``ergodic optimization'', see \cite{MR4000508}.
We include a proof for the reader's convenience.

\begin{proof}
     For a finite set $F\Subset \Gamma$ and $f:X \to [0,+\infty)$ denote
     \[
     M(F,f) = \sup_{x \in X}\sum_{\gamma \in F}f(T_\gamma(x))
     \]
     The existence of the limit on the right-hand-side of \Cref{eq:erg_max_avg_lim} is a consequence of the Ornstien-Weiss subadditivity lemma and subadditivity and $\Gamma$-invariance of the function 
     $F \mapsto M(F,f)$. 
     See \cite[Theorem 6.1]{MR1749670} for a precise statement and discussion.
     
     Furthermore, it is an relatively standard  exercise to check that  the function  $F \mapsto M(F,f)$ satisfies Shearer inequality in the sense of \cite[Definition 2.1]{MR3546663}, so by \cite[Proposition 3.3]{MR3546663} the expression on the right side of \Cref{eq:erg_max_avg_lim} is equal to the expression on the right side of \Cref{eq:erg_max_avg_inf}.

    Also, for any non-empty $F \in \Gamma$ any Borel $f:X \to [0,+\infty)$ and any $\mu \in \Prob(X,T)$ we have
    \[
    \int_X f(x)d\mu(x) = \int_X \frac{1}{|F|}\sum_{\gamma \in F}f(T_\gamma(x))d\mu(x) \le M(F,f),
    \]
    so 
    \[
    \sup_{\mu \in \Prob(X,T)}\int_X f(x)d\mu(x) \le \inf_{F \Subset \Gamma}M(F,f).
    \]

    This shows that the inequality ``$\le$'' holds in \Cref{eq:erg_max_avg_lim} and \Cref{eq:erg_max_avg_inf}.

    Now suppose that $f:X \to [0,\infty)$ is continuous.

    Take a sequence of points $(x_n)_{n=1}^\infty$ such that each $x_n$ maximizes the function 
    \[
    x \mapsto \frac{1}{|F_n|}\sum_{\gamma \in F_n}f(T_\gamma(x)).
    \]
    
    Let $\mu_n = \frac{1}{|F_n|}\sum_{\gamma \in F_n} \delta_{T_\gamma(x_n)} \in \Prob(X)$. By the choice of the points $x_n$ we have that $\int_X f(x)d\mu_n(x)=\frac{1}{|F_n|}M(F_n,f)$
    
    Passing to a subsequence using compactness of $\Prob(X)$, we can assume that
    the sequence $(\mu_n)_{n=1}^\infty$ converges and let $\mu = \lim_{n \to \infty}\mu_n$.

    Then 
    \[
    \int_X f(x) d\mu(x) = \lim_{n \to \infty}\int_X f(x) d\mu_n(x) =   \lim_{n \to \infty} \frac{1}{|F_n|}M_n(F_n,f).
    \]

    This shows that for any continuous $f:X \to [0,+\infty)$
    we have 
    \[
    \sup_{\mu \in \Prob(X,T)}\int_X f(x) d\mu(x) \ge \lim_{n \to \infty} \frac{1}{|F_n|}M_n(F_n,f),
    \]

    So \Cref{eq:erg_max_avg_lim} and \Cref{eq:erg_max_avg_inf} hold for any continuous $f:X \to [0,+\infty)$.

    Now let $f:X \to [0,+\infty)$ be upper semicontinuous. 
    To complete the proof of the proposition it remains to show that in this case we still have
    \begin{equation}\label{eq:max_meas_av_ge_max_inf_space_avg}
        \sup_{\mu \in \Prob(X,T)} \int_X f(x) d\mu \ge \inf_{F \Subset \Gamma} \frac{1}{|F|} M(F,f).
    \end{equation}

    Since $f$ is upper semicontinuous, exist a pointwise non-increasing sequence of continuous functions $f_m:X \to [0,+\infty)$  such that $f = \inf_{m \in \N} f_m$.

    So for every $m \in \N$
    \[
     \sup_{\mu \in \Prob(X,T)} \int_X f(x) d\mu \ge
      \sup_{\mu \in \Prob(X,T)} \int_X f_m(x) d\mu 
      = \inf_{F \Subset\Gamma} \frac{1}{F}M(F,f_m).
    \]
    Taking infimum over $m \in \N$, it follows that
    \[
    \sup_{\mu \in \Prob(X,T)} \int_X f(x) d\mu \ge 
    \inf_{m \in \N}\inf_{F \Subset \Gamma}\frac{1}{F}M(F,f_m).
    \]
    Note that
    \[
    \inf_{m \in \N}\inf_{F \Subset \Gamma}\frac{1}{F}M(F,f_m)=
    \inf_{F \Subset \Gamma}\inf_{m \in \N}\frac{1}{F}M(F,f_m)
    \]
    For every finite non-empty set $F \Subset \Gamma$ by monotonicity of  $f \mapsto \frac{1}{|F|}M(F,f)$ we have that 
    \[
    \inf_{m \in \N}\frac{1}{|F|}M(F,f_m) = \frac{1}{|F|}M(F,\inf_{m \in \N}f_m)= \frac{1}{|F|}M(F,f).
    \]
    We conclude the inequality \eqref{eq:max_meas_av_ge_max_inf_space_avg} holds.

\end{proof}

For any $\overline{\mathcal{U}} \in \overline{\Cov}(X)$, the function $x \mapsto \ord(\overline{\mathcal{U}},x)$ is upper semicontinuous so by \Cref{prop:upper_semi_erg_max_value} we have the following:
\begin{proposition}\label{prop:dim_amenable}
    Let $\Gamma$ be a finite or countable amenable group, let $(F_n)_{n=1}^\infty$ be a F\o lner sequence in $\Gamma$ and let $(X,T)$ be a $\Gamma$-dynamical system. Then for any $\overline{\mathcal{U}} \in \overline{\Cov(X)}$ the following holds:
    \begin{equation}\label{eq:ord_amenable_lim}
        \ord(\overline{\mathcal{U}},T) =\lim_{n \to \infty}\sup_{x \in X}\frac{1}{|F_n|}\sum_{\gamma \in F_n}\ord(\overline{\mathcal{U}},T_\gamma(x))
    \end{equation}
    and
    \begin{equation}\label{eq:ord_amenable_inf}
        \ord(\overline{\mathcal{U}},T) =\inf_{F \Subset \Gamma}\frac{1}{|F|}\sup_{x \in X}\sum_{\gamma \in F}\ord(\overline{\mathcal{U}},T_\gamma(x))
    \end{equation}
\end{proposition}

Consequently, if $(X,T)$ is a $\Gamma$-dynamical system and for some amenable group $\Gamma$,  using  for any $\mathcal{U} \in \Cov(X)$ we have
\[
\dim(\mathcal{U},T) = \inf_{\mathcal{V} \succeq \mathcal{U}} \inf_{F \Subset\Gamma}\frac{1}{|F|}\sup_{x \in X}\sum_{\gamma \in F}\ord(\mathcal{V},T_\gamma(x)),
\]
and for any F\o lner sequence $(F_n)_{n \in \N}$ of $\Gamma$ we have
\[
\dim(\mathcal{U},T) = \inf_{\mathcal{V} \succeq \mathcal{U}} \lim_{n \to \infty}\frac{1}{|F_n|}\sup_{x \in X}\sum_{\gamma \in F_n}\ord(\mathcal{V},T_\gamma(x)).
\]Using \Cref{lem:smaller_cover_munkres}, the infimum in the above equations can be taken either over all $\mathcal{V}$ in

\begin{theorem}\label{thm:mdim_le_dim}
       Let $\Gamma$ be a countable amenable group, and let $(X,T)$ be  a $\Gamma$-dynamical system. Then
    \[
    \mdim(X,T) \le \dim(X,T).
    \]
\end{theorem}

\begin{proof}
    We need to show that for every $t >0$,
    \[
    \dim(X,T) < t ~\Rightarrow~  \mdim(X,T) < t.
    \]
    Suppose that $\dim(X,T) <t$. Then for any $\epsilon >0$ there exists an open cover $\mathcal{V} \in \Cov(X)$ with $\mesh(\mathcal{V}) < \epsilon$ such that $\ord(\mathcal{V},T) < t$. This means that there exists a finite set $F \Subset \Gamma$ such that 
    \[
    \sup_{x \in X}\sum_{\gamma \in F}\ord(T_{\gamma}\mathcal{V},x) \le t |F|.
    \]
    By \Cref{lem:ord_subadditive} there exists $\mathcal{U} \in \Cov(X)$ such that
    $T_{\gamma}\mathcal{V} \preceq \mathcal{U}$ for all $\gamma \in F$ and so that for every $x \in X$
    \[
    \ord(\mathcal{U},x) \le \sum_{\gamma \in F}\ord(T_{\gamma}\mathcal{V},x) \le t |F|.
    \]
    Because $T_{\gamma}\mathcal{V} \preceq \mathcal{U}$ for all $\gamma \in F$, it follows that $\mesh(T_\gamma^{-1}(\mathcal{U})) < \epsilon$ for all $\gamma \in F$.
    As $\epsilon >0$ was arbitrary, this shows that $\mdim(X,T) < t$.
\end{proof}

We now recall the uniform Rokhlin property (URP), introduced in  \cite{MR4468009}.
A $\Gamma$-dynamical system $(X,T)$ has the \emph{uniform Rokhlin property (URP)} if for every finite $K \Subset \Gamma$ and every $\epsilon >0$ there exist open sets $V_1,\ldots,V_r \subseteq X$ and finite $(K,\epsilon)$-invariant subsets $S_1,\ldots,S_r \Subset \Gamma$ such that the collection of sets $\bigcup_{i=1}^r\{T_\gamma V_i\}_{\gamma \in S_i}$ are pairwise disjoint and 
\begin{equation}\label{eq:URP_castle_epsilon_full}
\sup_{\mu \in \Prob(X,T)}\mu \left( X \setminus \bigcup_{i=1}^r\bigcup_{\gamma \in S_i}T_{\gamma}V_i\right) < \epsilon. 
\end{equation}
It is obvious from the  definition that if a $\Gamma$-dynamical system $(X,T)$  has the uniform Rokhlin property then $\Gamma$ must be an amenable group. It follows from \cite{MR1793417} that any minimal $\Z$-dynamical system has the uniform Rokhlin property. See \cite[Lemma 3.6]{MR4468009} for details. More generally, the uniform Rokhlin property holds for any $\Z^k$-dynamical system that admits the so called \emph{marker property} \cite{MR4802735}. In particular, the class of system that admits the uniform Rokhlin property include any extension of a free minimal $\Z^k$-system.
We refer to  \cite{naryshkin2024urpcomparisonmeandimension} for further discussion of the uniform Rokhlin property and its significance.

Using standard arguments, we can replace the open sets in the definition of uniform Rokhlin property by closed sets. 
This is expressed by the lemma below (see eg. 
\cite{naryshkin2024urpcomparisonmeandimension} for a similar statement).
\begin{lemma}\label{lem:URP_closed_sets}
    Suppose that  $(X,T)$ is a $\Gamma$-dynamical system with the uniform Rokhlin property. Then for every finite $K \Subset \Gamma$ and every $\epsilon >0$ there exist \emph{closed} sets $V_1,\ldots,V_r \subseteq X$ and finite $(K,\epsilon)$-invariant subsets $S_1,\ldots,S_r \Subset \Gamma$ such that the collection of sets $\bigcup_{i=1}^r\{T_\gamma V_i\}_{\gamma \in S_i}$ are pairwise disjoint and \Cref{eq:URP_castle_epsilon_full} holds.
\end{lemma}
\begin{proof}
Choose a finite set  $K \Subset \Gamma$ and $\epsilon >0$.
Then the  uniform Rokhlin property $(X,T)$ implies that there exists open sets $U_r,\ldots,U_r \subseteq X$ and finite $(K,\epsilon)$-invariant subsets $S_1,\ldots,S_r \Subset \Gamma$ such that the collection of sets $\bigcup_{i=1}^r\{T_\gamma U_i\}_{\gamma \in S_i}$ are pairwise disjoint and 
\[
\sup_{\mu \in \Prob(X,T)}\mu \left( X \setminus \bigcup_{i=1}^r\bigcup_{\gamma \in S_i}T_{\gamma}U_i\right) < \epsilon. 
\]
By \Cref{lem:cap_continuous} we can find  $\delta >0$ such that 
\[
\sup_{\mu \in \Prob(X,T)}\mu \left(B_\delta( X \setminus \bigcup_{i=1}^r\bigcup_{\gamma \in S_i}T_{\gamma}U_i) \right) < \epsilon. 
\]
For $\eta >0 $ and $U \subseteq X$ let 
\[
\partial_\eta U =\{ x \in X~:~ B_\eta(x)\cap U \ne \emptyset \mbox { and } B_\eta(x) \cap (X\setminus U) \ne \emptyset\}.
\]
Then by (uniform) continuity of each $T_\gamma$ we can find $\eta >0$ such that for every $1\le i \le r$ and every $\gamma \in S_i$ 
\[
T_\gamma(\partial_\eta U_i) \subseteq B_\delta( X \setminus \bigcup_{i=1}^r\bigcup_{\gamma \in S_i}T_{\gamma}U_i).
\]
For every $1\le i \le d$ let $V_i = \overline{U_i \setminus \partial_\eta U_i}$. Then $V_1,\ldots, V_r$ are closed  such that the collection of sets $\bigcup_{i=1}^r\{T_\gamma V_i\}_{\gamma \in S_i}$ are pairwise disjoint and \Cref{eq:URP_castle_epsilon_full} holds. 
\end{proof}

\begin{theorem}\label{thm:mdim_equal_dim_marker_property}
    Let $\Gamma$ be a countable amenable group and let $(X,T)$ be a $\Gamma$-dynamical system with the uniform Rokhlin property.
    Then 
    $\mdim(X,T)=\dim(X,T)$.  
\end{theorem}
\begin{proof}
Let $(X,T)$ be a dynamical system with the uniform Rokhlin property.
By \Cref{thm:mdim_le_dim}, we only need to prove the inequality  $\dim(X,T)\le \mdim(X,T)$.
Suppose that $t \in \mathbb{R}$ and $\mdim(X,T) < t$.
We need to show that $\dim(X,T) <t$. It suffices to show that for any $\epsilon,\eta >0$ there exists a finite closed cover $\mathcal{V}$ of $X$ such that $\mesh(\mathcal{V}) < \epsilon$ and $\ord(\mathcal{V},T) < t+\eta$.
So let $\epsilon,\eta > 0$ be given.

Since $X$ is a compact metric space, there exist a sequence of continuous function $f_1,f_2,\ldots,f_n,\ldots:X \to [0,1]$ so that $\overline{f}=(f_1,\ldots,f_n,\ldots):X \to [0,1]^\mathbb{N}$ is a topological embedding.
    By compactness we can find $d \in \mathbb{N}$ such that  $f=(f_1,\ldots,f_d):X \to [0,1]^d$ satisfies that  the diameter of  $f^{-1}(\{x\})$ is less than  $\epsilon$ for any $x \in X$. Again by compactness of $X$ there exists $\delta>0$ such that the pre-image  of any $\delta$-Ball in $[0,1]^d$ has diameter less than $\epsilon$.
We can thus fix  $d \in \mathbb{N}$ and $\delta >0$ such that the set of continuous functions $f:X\to [0,1]^d$ with the property that the  inverse image under $f$ of any $\delta$-Ball in $[0,1]^d$ has diameter less than $\epsilon$ is a non-empty open set. 

Because $\mdim(X,T) < t$, for any $n \in \mathbb{N}$ there exist
$\delta_n >0$ and a finite set $K_n \Subset \Gamma$ such that for any $(K_n,\delta_n)$-invariant set $F_n \Subset \Gamma$ 
there exist an open cover $\mathcal{U}_n \in \Cov(X)$ such that $\mesh(T_\gamma (\mathcal{U}))<\frac{1}{n}$ for all $\gamma \in F_n$ and  
 \[
 \ord(\mathcal{U}) \le |F_n| t.
 \]

By the uniform Rokhlin property of $(X,T)$ and \Cref{lem:URP_closed_sets} for every $n \in \N$ there  exist closed sets $V_{1,n},\ldots, V_{r_n,n} \subset X$ and finite $(K_n,\delta_n)$-invariant subsets
$S_{1,n},\ldots,S_{r_n,n}$ so that $(T_\gamma V_{i,n})_{1\le i \le r_n,~ \gamma \in S_{i,n}}$ are pairwise disjoint and 
\[\sup_{\mu \in \Prob(X,T)} \mu\left( X\setminus \bigcup_{i=1}^{r_n}\bigcup_{\gamma \in S_{i,n}}T_\gamma V_{i,n} \right)<\frac{1}{n}.\]

By the choice of $(K_n,\delta_n)$, the fact that the sets $S_{i,n}$ are  $(K_n,\delta_n)$-invariant for every $1\le i \le r_n$ there exists an open cover $\mathcal{U}_{i,n} \in \Cov(X)$ such that  $\mesh(T_\gamma (\mathcal{U}_{i,n}))<\frac{1}{n}$ for all $\gamma \in S_i$ and  
 \[
 \ord(\mathcal{U}_{i,n}) \le |S_{i,n}| t.
 \]

By \Cref{lem:ostrand_kolmogorov_covers},  for every $1\le i \le d$
there exists families $\{\mathcal{C}_{i,n,\gamma,\ell}\}_{ \gamma \in S_i,~ 1\le \ell \le d}$ such that each family $\mathcal{C}_{i,n,\gamma,\ell}$ consists of pairwise disjoint closed subset such that 
$\mathcal{U}_{i,n} \preceq \mathcal{C}_{i,n,\gamma,\ell}$

and so that for every $x \in X$ and every $1\le i \le r$

\begin{equation}\label{eq:C_i_gamma_ell_bound}
\sum_{\ell=1}^d \sum_{\gamma \in S_i} \overline{\mathbf{1}}_{{C}_{i,n,\gamma, \ell}}(x) \le   |S_{i,n}|t.  
\end{equation}

Because  $\mesh(T_\gamma (\mathcal{U}_{i,n}))<\frac{1}{n}$ for all $\gamma \in S_{i,n}$, it follows that $\mesh( T_g(\mathcal{C}_{i,n,\gamma, \ell})) < \frac{1}{n}$ for every $g \in S_{i,n}$.
 
Consider the collection of sets 
\[
\mathcal{D}_n = \bigcup_{i=1}^r\bigcup_{\gamma \in S_{i,n}}\bigcup_{\ell=1}^d \left\{ T^{-1}_{\gamma}C \cap T^{-1}_\gamma(V_i) ~:~ C \in \mathcal{C}_{i,n,\gamma,\ell} \right\}.
\]
Then  $\mathcal{D}_n$ is a collection of pairwise disjoint closed sets with $\mesh(\mathcal{D}_n)< \frac{1}{n}$.
Also, for any $x \in V_i$ we have that
\begin{equation}\label{eq:ord_D_n}
    \sum_{\gamma \in S_{i,n}}\overline{\mathbf{1}}_{\mathcal{D}_n}(T^{-1}_\gamma(x)) = \sum_{\ell=1}^d\sum_{\gamma \in S_{i,n}}\overline{\mathbf{1}}_{\mathcal{C}_{i,n,\gamma,\ell}}(x).
\end{equation}

By \Cref{lem:cube_cover_R_d}  
there exists a cover $\mathcal{P}$ of $[0,1]^d$ by closed 
set with $\mesh(\mathcal{P})< \delta$ and pairwise disjoint, finite subsets $(A_{\ell})_{1 \le \ell \le d}$ such 
that for every $v \in [0,1]^d$ we have
\begin{equation}\label{eq:ord_P_upper_bound}
\ord(\mathcal{P},v) \le \sum_{\ell=1}^d\mathbf{1}_{A_{\ell}}(v).    
\end{equation}

Let $A =\bigcup_{\ell=1}^dA_{\ell,m}$.
Then $A \subset [0,1]$ is a finite set.
By \Cref{lem:separate} for  there exists a dense $G_\delta$ set of continuous functions $f=(f_1,\ldots,f_d):X \to [0,1]^d$ such that for infinitely many $n$'s we have that 
$f_i(x) \not \in A$ for all $x \in \bigcup_{D \in \mathcal{D}_n}D$.
Because the for every $m$ the sets $A_{1},\ldots,A_{d}$ are pairwise disjoint, the above argument means that there is a dense set $G_\delta$ of continuous functions $f:X \to [0,1]^d$ such that for infinitely many $n$'s  
\begin{equation}\label{eq:sum_A_ell_le_ord_D_n}
    \sum_{\ell=1}^d \mathbf{1}_{A_{\ell}}(f_\ell(x)) \le \overline{\mathbf{1}}_{\mathcal{D}_{n}}(x) \mbox{ for all } x \in X.
\end{equation} 
By our choice of $d$ and $\delta>0$, there is a non-empty open set of functions so that furthermore the  inverse image under $f$ of any $\delta$-Ball in $[0,1]^d$ has diameter less than $\epsilon$.

We can thus find a continuous function $f:X \to [0,1]^d$ that satisfies the above and also  \ref{eq:sum_A_ell_le_ord_D_n} for infinitely many $n$'s and so that that .
Let $\mathcal{A}=f^{-1}(\mathcal{P})$. It follows that $\mathcal{A} \in \overline{Cov}(X)$ and $\mesh(\mathcal{A}) < \epsilon$.

Choose any $\mu \in \Prob(X,T)$.
Our goal is to show that $\int_X \ord(\mathcal{A},x) d\mu(x)\le t$ 
for every $n \in \N$ let  \[R_n= X \setminus \bigcup_{i=1}^{r_n}\bigcup_{\gamma \in S_{i,n}}T_\gamma(V_{i,n}).\]
Because the sets $(T_\gamma V_{i,n})_{1 \le i \le r_n,\gamma \in S_{i,n}}$ are pairwise disjoint we have:
\[
\int_X \ord(\mathcal{A},x) d\mu(x) = \int_{R_n} \ord(\mathcal{A},x) d\mu(x) + \sum_{i=1}^{r_n} \sum_{\gamma \in S_{i,n}}\int_{T_\gamma S_i}\ord(\mathcal{A},x) d\mu(x).
\]
By \eqref{eq:ord_P_upper_bound}  for every $x \in X$ we have 
\begin{equation}\label{eq:ord_A_le_sim_A_ell}
\ord(\mathcal{A},x) \le \sum_{\ell=1}^d\mathbf{1}_{A_{\ell}}(f_\ell(x)).     
\end{equation}

In particular, $\ord(\mathcal{A},x) \le d$ for any $x \in X$,
so
\[
\int_{R_n} \ord(\mathcal{A},x) d\mu(x) \le d\cdot \capacity(R_n,T) \to 0 \mbox{ as } n\to +\infty.
\]

We will conclude the proof by showing that there exists infinitely many $n$'s so that  
\[
\sum_{i=1}^{r_n} \sum_{\gamma \in S_{i,n}}\int_{T_\gamma V_i}\ord(\mathcal{A},x) d\mu(x) \le t, 
\]

Because  $\mu$ is $T$-invariant we have:
\[
\sum_{i=1}^{r_n} \sum_{\gamma \in S_{i,n}}\int_{T_\gamma V_i}\ord(\mathcal{A},x) d\mu(x)=
\sum_{i=1}^{r_n}\int_{V_i} \sum_{\gamma \in S_{i,n}}\ord(\mathcal{A},T_\gamma^{-1}(x))d\mu(x).
\]
By \eqref{eq:ord_A_le_sim_A_ell} and \eqref{eq:sum_A_ell_le_ord_D_n} there are infinitely many $n$'s so that  for any $1 \le i \le r_n$ and $x \in V_i$ we have
\[
\sum_{\gamma \in S_{i,n}}\ord(\mathcal{A},T_\gamma^{-1}(x)) \le
\sum_{\gamma \in S_{i,n}}\sum_{\ell=1}^d \mathbf{1}_{A_\ell}(f(T_\gamma(x))_\ell) \le
\sum_{\ell=1}^d\overline{\mathbf{1}}_{\mathcal{D}_n}(T_\gamma^{-1}(x)).
\]
So using \Cref{eq:ord_D_n}  there are infinitely many $n$'s so that  for every $x \in V_{i,n}$
\[
\sum_{\gamma \in S_{i,n}}\ord(\mathcal{A},T_\gamma^{-1}(x)) \le
\sum_{\ell=1}^d\sum_{\gamma \in S_{i,n}}\overline{\mathbf{1}}_{\mathcal{C}_{i,n,\gamma,\ell}}(x).\]
By \Cref{eq:C_i_gamma_ell_bound} we thus have that there are infinitely many $n$'s so that for any $x \in V_{i,n}$
\[
\sum_{\gamma \in S_{i,n}}\ord(\mathcal{A},T_\gamma^{-1}(x)) \le |S_{i,n}|t.
\]
It follows that there are infinitely many $n$'s so that
\[
\int_{V_{i,n}} \sum_{\gamma \in S_{i,n}}\ord(\mathcal{A},T_\gamma^{-1}(x))d\mu(x) \le |S_{i,n}|t \mu(V_{i,n}).
\]
Summing over $1\le i \le r_n$ we get
\[
\sum_{i=1}^{r_n} \int_{V_{i,n}} \sum_{\gamma \in S_{i,n}}\ord(\mathcal{A},T_\gamma^{-1}(x))d\mu(x) \le t \left(\sum_{i=1}^{r_n} |S_{i,n}|\mu(V_{i,n})\right) \le t,
\]
where in the right inequality we used that $(T_{\gamma}(V_{i,n}))_{1\le i \le r_n,\gamma \in S_{i,n}}$ are pairwise disjoint.
\end{proof}

\section{An almost embedding theorem}\label{sec:almost_emb}

In this section we prove a weak form of shift embeddability, assuming only that $\dim(X,T) < +\infty$. This result will allow us to deduce some interesting positive results regarding the classical shift embeddability problem.

Let $(X,T)$ and $(Y,S)$ be $\Gamma$-dynamical systems. An \emph{almost  embedding} is a  equivariant map $f:X \to Y$ with the property that for every $\mu \in \Prob(X,T)$ the function $f:X\to Y$ induces a measure-theoretic isomorphism of $(X,T,\mu)$ and $(Y,S,f_* \mu)$.  In this case we say that that $(X,T)$ \emph{almost embeds} into $(Y,S)$. 
In the terminology of Downarowicz and Glasner \cite{MR3586277}, if $(X,T)$ almost embeds into $(Y,S)$, then $(X,T)$ is an \emph{isomorphic extension} of its image in $Y$. In the terminology of Kerr and Szab\'{o} \cite{MR4066584}, in this situation $(X,T)$ is \emph{measure-isomorphic} to its image in $Y$. 

To formulate and prove an alternative   characterization of almost embeddings, we now recall the notion of the fiber product (sometimes also called relative product). 
Given an equivariant continuous map $f:X \to Y$ between dynamical systems    $(X,T)$ and $(Y,S)$, the  \emph{fiber product} of $X$ over $f$ is given by
\[
X \times_f X= \left\{ (x_1,x_2) \in X \times X : f(x_1)=f(x_2) \right\}.
\]
Evidently,  $X \times_f X$ is a closed, $T \times T$-invariant subset of $X \times X$.

We always have $\Delta_X \subseteq X \times_f X$, where
\[
\Delta_X = \left\{
(x,x) ~:~ x \in X
\right\}.
\]

The following result gives alternative characterization of almost  embeddings. 

\begin{proposition}\label{prop:almost_emb_char}
    Let $(X,T)$ and $(Y,S)$ be $\Gamma$-dynamical systems and let  $f:X \to Y$ be a continuous equivariant map. The following are equivalent:
    \begin{enumerate}
        \item $f:X \to Y$ is an almost embedding.
        \item Any $(T \times T)$-invariant probability measure on the fiber product $X \times_f X$ is supported on the diagonal $\Delta_X \subseteq X \times_f X$. 
    \end{enumerate}
\end{proposition}
\begin{proof}
    Suppose first that $f:X \to Y$ is not an almost  embedding. Then there exists  $\mu \in \Prob(X,T)$ such that $f:X \to Y$ is not a measure theoretic  isomorphism  from $(X,T,\mu)$ to $(Y,S,f_* \mu)$. Let $\lambda = \mu \times_f \mu$ be the relatively independent self-joining of $\mu$ given $f$, which is (essentially by definition) the unique Borel probability measure supported on $X \times_f X$ satisfying that $(f \times f)_* \lambda =  f_* \mu $ together with the conditional independence property that for $\lambda$-almost every $(x_1,x_2) \in X \times_f X$ we have
    \[
    \lambda\left( A \times B \mid f\times f\right)(x_1,x_2) = \mu(A \mid f)(x_1)\mu(B\mid f)(x_2).
    \]
    Then $\lambda \in \Prob(X\times_f X, T\times T)$.
    The assumption that $f:X \to Y$ is not a  measure theoretic  isomorphism  from $(X,T,\mu)$ to $(Y,S,f_* \mu)$ directly implies that $\lambda$ is not supported on $\Delta_X$.

    Conversely, suppose that $\lambda \in \Prob(X \times_f X,T \times T)$ is not supported  on $\Delta_X$. Let $\mu_1, \mu_2 \in \Prob(X \times_f X, T\times T)$ denote the push-forward of $\lambda$ via the $\pi_1:X \times_f X \to X$ and $\pi_2:X \times X \to X$ respectively where $\pi_i(x_1,x_2)=x_i$, and let
    $\mu = \frac{1}{2} \mu_1 + \frac{1}{2}\mu_2$. Then statement that $\lambda$ is not supported on $\Delta_X$ implies in particular that on a set of positive $\lambda$-measure the conditional measure $\lambda(\cdot \mid f\times f)$ is supported on a singleton $(x,x) \in \Delta_X$. This directly implies that on a set of positive $\mu$-measure $\mu(\cdot \mid f)$ is not supported on a singleton, so $f:X \to Y$ is not a measure theoretic  isomorphism  from $(X,T,\mu)$ to $(Y,S,f_* \mu)$.
\end{proof}

We note that a version of \Cref{prop:almost_emb_char} for the case where $(Y,S)$ is uniquely ergodic appears in \cite[Proposition 2.5]{MR3586277}. In the case where $(Y,S)$ is uniquely ergodic condition (2) in the statement of \Cref{prop:almost_emb_char} can be replaced by the simpler condition that $X \times_f X$ is uniquely ergodic.

There are  simple examples for continuous  almost embeddings which are not an embedding. For example, take $(X,T)$ to be the $\mathbb{Z}$-dynamical system given by  $X= \Z \cup\{\infty\}$ is the one-point compactification of $\Z$ and $T:X \to X$ is given by \[T(x)=\begin{cases}
    x+1 & x \in \mathbb{Z}\\
    \infty & x=\infty
\end{cases},\]
Then the unique map from $X$ to the trivial one-point space is an almost embedding. More generally, any proximal extension induces an almost embedding.

However, for the class of \emph{distal} dynamical systems any almost embedding is an embedding. 
A dynamical system $(X,T)$ is \emph{distal} if for every pair of distinct point $(x_1,x_2) \in (X\times X) \setminus \Delta_X$, the orbit closure of $(x_1,x_2)$ under $T\times T$ does not intersect $\Delta_X$.

\begin{proposition}\label{prop:almost_emb_distal}
    Let $(X,T)$ be a distal dynamical system. Then any almost  embedding is an embedding.
\end{proposition}
\begin{proof}
    Suppose that $(X,T)$ is a distal $\Gamma$-dynamical system for some group  $\Gamma$ and that $f:X \to Y$ is an almost embedding.  To show that $f:X \to Y$ is an embedding we need to show that $X \times_f X = \Delta_X$. Take any $(x_1,x_2) \in X \times_f X$. 
    Since the class of distal dynamical system is closed under products and taking subsystems, it follows that $(X\times_f X ,T \times T)$ is distal. Thus, the orbit closure of $(x_1,x_2)$ under $T\times T$ is also distal. By a celebrated result of Furstenberg, any distal system admits an invariant probability measure \cite[Theorem 12.3]{MR157368}. In particular, the orbit closure of $(x_1,x_2)$ supports a $T \times T$-invariant probability measure $\lambda$. The assumption that $f:X \to Y$ is an almost-embedding implies that $\lambda$ is supported on $\Delta_X$. But the assumption that $(X,T)$ is distal means that the orbit closure of $(x_1,x_2)$ intersects $\Delta_X$ if  and only if $(x_1,x_2) \in \Delta_X$. We conclude that $X \times_f X=\Delta_X$.
\end{proof}

\begin{theorem}\label{thm:dim_almost_emb}
    Let $\Gamma$ be a countable group, and let $(X,T)$ be a $\Gamma$-dynamical system and let $d \in \N$ be an number such that $\dim(X,T) < \frac{d}{2}$. 
    Then there is a dense $G_\delta$ set $G \subseteq C(X,[0,1]^d)$ such that for every $f \in G$ the function $f^\Gamma:X \to ([0,1]^d)^\Gamma$ is an almost embedding. 
 
\end{theorem}
\begin{proof}
Because $\dim(X,T) < \frac{d}{2}$, we can choose $0 < \eta < \frac{d}{2} -\dim(X,T)$. For any $n \in \mathbb{N}$ there exists a finite open cover $\mathcal{U}_n \in \Cov(X)$ such that $\mesh(\mathcal{U}_n) < \frac{1}{n}$ and $\ord(\mathcal{U}_n,T) < \frac{d}{2}-\eta$. Using \Cref{lem:ostrand_kolmogorov_covers}, for any $n \in \mathbb{N}$ there exists families $\mathcal{C}_{1,n},\ldots,\mathcal{C}_{d,n}$ such that each $\mathcal{C}_{j,n}$ consists of pairwise disjoint closed subsets of $X$, $\mesh(\mathcal{C}_{j,n}) < \frac{1}{n}$ and 
\begin{equation}\label{eq:sum_C_j_int_mu}
\inf_{\mu \in \Prob(X,T)} \int \sum_{j=1}^d \overline{\mathbf{1}}_{\mathcal{C}_{j,n}}(x)d\mu(x) < \frac{d}{2}- \eta.
\end{equation}
By \Cref{lem:separate} the set $G \subseteq C(X,[0,1]^d)$ of functions $f=(f_1,\ldots,f_d):X \to [0,1]^d$ such that for infinitely many $n$'s   $f_j$ separates $\mathcal{C}_{j,n}$ for all $1\le j \le d$ is a dense $G_\delta$ subsets in $C(X,[0,1]^d)$. Choose $f \in G$.
We will prove that $f^\Gamma$ is an  almost embedding.
Choose any ergodic $\lambda \in \Prob(X\times_f X,T \times T)$. Our goal is to show that $\lambda(\Delta_X)=1$.  Since $\lambda$ is ergodic, and $\Delta_X$ is $(T\times T)$-invariant, it suffices to show that $\lambda(\Delta_X) >0$.  Note that
\[\Delta_X = \bigcap_{n=1}^\infty\left\{ (x_1,x_2) \in X \times_f X~:~ \rho(x_1,x_2) < \frac{1}{n} \right\}\]
So in order to show that $\lambda(\Delta_X)>0$, it suffices to find a positive number $\alpha >0$ so that the following inequality holds for infinitely many $n$'s:
\begin{equation}\label{eq:lambda_dist_small_n}
\lambda\left(\left\{(x_1,x_2)\in X \times_f X~:~ \rho(x_1,x_2) < \frac{1}{n} \right\}\right) > \alpha .   
\end{equation}


Since $\overline{\mathbf{1}}_{\mathcal{C}_{j,n}}=1- \mathbf{1}_{\mathcal{C}_{j,n}}$ it follows from \eqref{eq:sum_C_j_int_mu} that

\[
\int \sum_{j=1}^d\left( \mathbf{1}_{\mathcal{C}_{j,n}}(x_1) + \mathbf{1}_{\mathcal{C}_{j,n}}(x_2)\right) d\lambda(x_1,x_2) > d + 2\eta. 
\]
It follows that for every $n \in \N$
\[
\lambda \left( \bigcup_{j=1}^n\bigcup_{C_1,C_2 \in \mathcal{C}_{n,j}}\left\{ (x_1,x_2) \in X \times_f X~:~ (x_1,x_2) \in  C_1 \times C_2 \right\}\right) > 2\eta. 
\]
Now if $f=(f_1,\ldots,f_j)$ and $f_j$ separates $\mathcal{C}_{j,n}$ and $(x_1,x_2) \in X \times_f X$ satisfies that $(x_1,x_2) \in C_1 \times C_2$
for some $C_1,C_2 \in \mathcal{C}_{j,n}$
then there exists some $C_1 = C_2$ and because $\mesh(\mathcal{C}_{j,n}) < \frac{1}{n}$ it follows that $\rho(x_1,x_2) < \frac{1}{n}$.
Thus, we see that for any $f \in G$ the inequality \eqref{eq:lambda_dist_small_n} holds with $\alpha =2\eta$, hence any $f \in G$ is an almost embedding.

\end{proof}

Combining \Cref{prop:almost_emb_distal} and \Cref{thm:dim_almost_emb} 
we have:

\begin{corollary}\label{cor:distal_emb}
Suppose that $(X,T)$ is a distal dynamical system with $\dim(X,T) < \frac{d}{2}$. Then $(X,T)$ embeds in $(((0,1)^d)^\Gamma,\shift)$.    
\end{corollary}

As mentioned in the introduction, Dranishnikov  and Levine  proved the existence of a free and equicontinuous $\mathbb{Z}$-dynamical system that cannot embed into $(((0,1)^d)^\mathbb{Z},\shift)$ for any $d \in \mathbb{N}$ in \cite{dranishnikov2025freezactionisometriescompact}. It follows from \Cref{cor:distal_emb} that any such dynamical system $(X,T)$ must have $\dim(X,T)=\infty$. We conclude that $\dim(X,T)$ does not coincide with the mean dimension of $(X,T)$ even for free actions of $\mathbb{Z}$:

\begin{corollary}\label{cor:dim_ne_mdim}
    There exists a free $\mathbb{Z}$-dynamical system $(X,T)$ with $\mdim(X,T)=0$ and $\dim(X,T)=+\infty$.
\end{corollary}

\bibliographystyle{amsplain}
\bibliography{lib}
\end{document}